\documentclass[10pt]{article}
\usepackage[ansinew]{inputenc}
\usepackage{amsmath}
\usepackage{amsfonts}
\usepackage{amssymb}
\usepackage{amsthm}
\usepackage{newlfont}
\usepackage{amsbsy}
\usepackage{bm}
\usepackage{color}
\usepackage{epsfig}
\usepackage{dsfont}
\usepackage{latexsym}
\usepackage{array}
\usepackage{longtable}
\usepackage{enumerate}
\usepackage{multicol}
\usepackage{mathrsfs}
\usepackage{wasysym}
\usepackage{graphics}
\usepackage{graphicx}
\usepackage{hyperref}
\hypersetup{
    colorlinks,
    citecolor=blue,
    filecolor=black,
    linkcolor=red,
    urlcolor=black
}

\newtheorem{lemma}{Lemma}[section]
\newtheorem{theorem}{Theorem}[section]

\newtheorem{prop}{Proposition}[section]
\newtheorem{rem}{\it Remark}[section]
\newcommand{\nor}[2]{{\left\|{#1}\right\|_{#2}}}
\newcommand{\nora}[3]{{\left\|{#1}\right\|_{#2}^{#3}}}

\title{\textbf{Sharp well-posedness for the Chen-Lee equation}}
\author{Ricardo A. Pastr\'an R. \thanks{Ricardo Pastr\'an. Departamento de Matem\'aticas, Universidad Nacional de Colombia, Bogot\'a, Colombia. Carrera 30 No. 45-03. C\'odigo postal 11321.
E-mail: {\tt rapastranr@unal.edu.co}}\\ Oscar G. Ria\~no C. \thanks{Departamento de Matem\'aticas, Universidad Nacional de Colombia, Bogot\'a, Colombia. Carrera 30 No. 45-03. C\'odigo postal 11321.
E-mail: {\tt ogrianoc@unal.edu.co}}}

\begin{document}

\maketitle

\begin{abstract}
We study the initial value problem associated to a perturbation of the Benjamin-Ono equation or Chen-Lee equation. We prove that results about local and global well-posedness for initial data in $H^s(R)$, with $s>-1/2$, are sharp in the sense that the flow-map data-solution fails to be $C^3$ in $H^s(\mathbb{R})$ when $s<-\frac{1}{2}$. Also, we determine the limiting behavior of the solutions when the dispersive and dissipative parameters goes to zero. In addition, we will discuss the asymptotic behavior (as $|x|\to \infty$) of the solutions by solving the equation in weighted Sobolev spaces. 
\end{abstract}


\section{Introduction}
This paper is concerned with the following initial value problem associated to a perturbation of the Benjamin-Ono equation or Chen-Lee equation
\begin{equation}\label{cl}
\text{CL} \left\{
\begin{aligned}
u_t+uu_x+\beta \mathcal{H}u_{xx} + \eta (\mathcal{H}u_x-u_{xx})&=0, \qquad x\in \mathbb{R} \;\,(\text{or}\;\, x\in \mathbb{T}), \quad t > 0,  \\
u(x,0)&=\phi(x),
\end{aligned}
\right.
\end{equation}
where $\beta,\,\eta >0$ are constants. In the equation, $\mathcal{H}$ denotes the usual Hilbert transform given by
\begin{equation*}
\mathcal{H}f(x)=\dfrac{1}{\pi}\,\text{p.v.}\int_{-\infty}^{\infty}\dfrac{f(y)}{y-x}\,dy=i\, (\text{sgn}(\xi)\widehat{f}(\xi))^{\vee}(x)\;\text{for}\; \xi \in \mathbb{R},\; f\in \mathcal{S}(\mathbb{R}).
\end{equation*}
This equation was first introduced by H. H. Chen and Y. C. Lee in \cite{CL} to describe fluid and plasma turbulence and as a model for internal waves in a two-fluid system. The fourth and the fifth terms represent the instability and dissipation, respectively. The parameter $\eta$ represents the importance of instability and dissipation relative to dispersion and nonlinearity. H. H. Chen, Y. C. Lee and S. Qian in \cite{clq, clq1}, and B. -F. Feng and T. Kawahara, in \cite{FeKa}, investigated the initial value problem as well as stationary solitary and periodic waves, associated with Chen-Lee equation, from a numerical standpoint. R. Pastr\'an in \cite{P} proved using the Fourier restriction norm method that the initial value problem CL is locally well-posed in $H^s(\mathbb{R})$ for any $s>-1/2$, globally well-posed in $H^s(\mathbb{R})$ when $s\geq 0$ and ill-posed in $H^s(\mathbb{R})$, if $s<-1$. Additionally, Pastr\'an and Ria\~no in \cite{PR} showed using the purely dissipative methods of Dix for the KdV-B equation \cite{Dix} that CL is locally and globally well-posed in the spaces $H^s(\mathbb{R})$ and $H^s(\mathbb{T})$ for any $s>-1/2$. In the periodic setting it was showed that CL is ill-posed in $H^s(\mathbb{T})$, when $s<-1$. 
\\ \\
Here, we say that the Cauchy problem or initial value problem CL is \textit{ill-posed} in the space $X$ when the flow-map data-solution fails to be $C^k$ in $X$ for some $k\in \mathbb{N}$. As a consequence, we cannot solve the Cauchy problem for the Chen-Lee equation by a Picard iterative method implemented on its integral formulation (see \cite{Amin, Pilod} for similar results). In particular, the methods introduced by Bourgain \cite{Bourgain} and Kenig, Ponce and Vega \cite{KPV} for the KdV equation cannot be used for CL with initial data in the space $X$. This kind of ill-posedness result is weaker than the loss of uniqueness proved by Dix in the case of Burgers equation \cite{Dix}. 
\\ \\
We begin showing that our results in \cite{PR} about local and global well-posedness in $H^s(\mathbb{R})$ with $s>-1/2$ are sharp in the sense that the flow map data-solution of the CL equation fail to be $C^3$ in $H^s(\mathbb{R})$ for $s<-1/2$. This result is equivalent to the fact that we cannot solve the Cauchy problem CL in $H^s(\mathbb{R})$, $s<-1/2$, using a contraction argument on the integral equation. Next, we will prove as in \cite{BI} that the solutions of the Chen-Lee equation when the dispersion $\beta$ tends to zero and the dissipation $\eta>0$ is fixed converge to the solutions of the non-dispersive Chen-Lee equation $(\beta=0)$ in the $C([0,T];H^s(\mathbb{R}))$ topology when $s>-1/2$ and, in the same way, the result in \cite{PR} for non-dispersive Chen-Lee equation is sharp in the sense that the flow-map fails to be $C^2$ in $H^s(\mathbb{R})$ when $s<-\frac{1}{2}$. Also, we are interested in the limiting behavior of the solutions of CL when the dissipative parameter $\eta$ tends to zero. It will be shown as in \cite{BI} that solutions of the CL equation tend to solutions of the BO equation in the $C([0,T];H^s(\mathbb{R}))$ topology when $\eta$ goes to zero and $s>3/2$. Finally, some decay properties of the solution of initial value problem CL for $\eta>0$ are obtained, similar to those obtained for the Benjamin-Ono equation (see \cite{Iorio}). More precisely, we will prove that if the solution $u(t)$ of CL is sufficiently smooth $(u(t)\in H^3(\mathbb{R}))$ and falls off sufficiently fast as $|x|\to \infty$ $(u(t)\in L_3^2(\mathbb{R}))$ for all $t\in [0,T]$, then $u(t)=0$, for all $t\in [0,T]$.

\subsection{Notation}
Given $a$, $b$ positive numbers, $a\lesssim b$ means that there exists a positive constant $C$ such that $a\leq C b$. And we denote $a\sim b$ when, $a \lesssim  b$ and $b \lesssim a$. We will also denote $a\lesssim_{\lambda} b$ or $b\lesssim_{\lambda} a$, if the constant involved depends on some parameter $\lambda$. Given a Banach space $X$, we denote by $\nor{\cdot}{X}$ the norm in $X$. We will understand $\langle \cdot \rangle = (1+|\cdot|^2)^{1/2}$.
\\ \\
 Let $U$ be the unitary group in $H^s(\mathbb{R})$, $s\in \mathbb{R}$, generated by the skew-symmetric operator $-\beta \mathcal{H}\partial_x^2 $, which defines the free evolution of the Benjamin-Ono equation, that is,
\begin{equation}\label{unitarygroup}
U(t)=\exp (itq(D_x)), \qquad U(t)f= \Bigl(e^{itq(\xi)}\Hat{f}\Bigr)^{\vee}\quad \text{with}\quad f\in H^s(\mathbb{R}), \, t\in \mathbb{R},
\end{equation}
where $q(D_x)$ is the Fourier multiplier with symbol $q(\xi)=\beta \,\xi \,|\xi|$, for all $\xi\in \mathbb{R}$. Since the linear symbol of equation in (\ref{cl}) is $iq(\xi)+p(\xi)$, where $p(\xi)=\eta \,(\xi^2-|\xi|)$ for all $\xi \in \mathbb{R}$, we also denote by $S(t)=e^{-(\beta \mathcal{H}\partial_x^2 +\eta (\mathcal{H}\partial_x-\partial_x^2))t}$, for all $t\geq 0$, the semigroup in $H^s(\mathbb{R})$ generated by the operator $-(\beta \mathcal{H}\partial_x^2 +\eta (\mathcal{H}\partial_x-\partial_x^2))$, i.e.,
\begin{align}
S(t)f=\Bigl( e^{i\,q(\xi)\,t -p(\xi)\,t}\Hat{f} \Bigr)^{\vee}\quad \text{for}\quad f\in H^s(\mathbb{R}),\;t\geq 0. \label{semigrupos}
\end{align}
We will employ weighted Sobolev spaces defined by
\begin{equation}\label{sobpeso}
\mathcal{F}_{s,r}=H^s(\mathbb{R})\cap L_r^2(\mathbb{R}),\qquad s, r=0,1,2,...\quad \text{and} \quad \nora{f}{\mathcal{F}_{s,r}}{2}=\nora{f}{H^s}{2}+\nora{f}{L_r^2}{2}.
\end{equation}
Here $L_r^2(\mathbb{R})$, $r\in \mathbb{R}$ is the collection of all measurable functions $f:\mathbb{R}\to \mathbb{C}$ such that
\begin{equation}\label{ldospeso}
\nora{f}{L_r^2}{2}=\int_{\mathbb{R}}(1+x^2)^r\,|f(x)|^2\,dx <\infty.
\end{equation}
\subsection{Main Results}
First, we recall the results about well-posedness in \cite{PR}.
\begin{theorem}[Local and Global well-posedness \cite{PR}]\label{mainresult}
Let $s>-1/2$, $\beta \geq 0$ and $\eta>0$. Then for any $\phi \in H^s(\mathbb{R})$ there exist $T=T(\nor{\phi}{H^s})>0$ and a unique solution $u$ of the integral equation (\ref{intequation}) satisfying
\begin{align*}
&u\in C([0,T],H^s(\mathbb{R}))\cap C((0,T),H^{\infty}(\mathbb{R})).
\end{align*}
Moreover, the flow map $\phi \mapsto u(t)$ is smooth from $H^s(\mathbb{R})$ to $C([0,T],H^s(\mathbb{R}))\cap C((0,T],H^{\infty}(\mathbb{R}))\cap X_T^s$. Additionally, the supremum of all $T>0$ for which all the assertions above hold is infinity.
\end{theorem}
The result of Theorem \ref{mainresult} is sharp in the following sense.
\begin{theorem}\label{malpuestodos}
Fix $s<-\frac{1}{2}$. Then there does not exist a $T>0$ such that (\ref{cl}) admits a unique local solution defined on the interval $[0,T]$ and such that the flow-map data-solution
\begin{equation}
\phi \longmapsto u(t), \qquad t\in [0,T],
\end{equation}
for (\ref{cl}) is $C^3$ differentiable at zero from $H^s(\mathbb{R})$ to $C\left([0,T];H^s(\mathbb{R})\right)$.
\end{theorem}
A direct corollary of Theorem \ref{malpuestodos} is the next statement.
\begin{theorem}\label{illposed}
The flow map in the existing results for the Chen-Lee equation is not $C^3$ from $H^s(\mathbb{R})$ to $C\left( [0,T],H^s(\mathbb{R})\right)$, if $s<-\frac{1}{2}$.
\end{theorem}
When the dispersive parameter $\beta$ is zero and $\eta>0$, we have the following Cauchy problem associated to the Chen-Lee non-dispersive equation
\begin{equation}\label{CCLequa1}
\text{CLND}\left\{\begin{aligned}
&u_t+ uu_x+\eta(\mathcal{H}u_x-u_{xx})=0, \\
&u(x,0) =\phi(x),
\end{aligned} \right.
\end{equation} 
$\phi \in H^s(\mathbb{R})$, $s\in\mathbb{R}$. Following the ideas presented by Vento in \cite{Vento} for the Dissipative Benjamin-Ono equation, we can prove that for the Chen-Lee non-dispersive equation \eqref{CCLequa1}, the result obtained in Theorem \ref{mainresult} is sharp with the next Theorem.

\begin{theorem}\label{malpuestoCLND}
Let $s<-\frac{1}{2}$ be given. Then there does not exist a time $T>0$ such that \eqref{CCLequa1} admits a unique local solution on the time interval $[0,T]$ and such that the flow map data-solution $\phi \longmapsto u(t)$
of \eqref{CCLequa1} is $C^2$ at the origin from $H^s(\mathbb{R})$ to $C\left([0,T];H^s(\mathbb{R})\right)$.
\end{theorem}

We study the limiting behavior of the solutions CL when $\eta>0$ is fixed and $\beta$ tends to zero with the next theorem.
\begin{theorem}\label{convergeCLND}
Let $\eta >0$ be fixed. Let $s > -1/2$ and $\phi\in H^s(\mathbb{R})$. If $u^{\beta}$ is the solution of equation \eqref{cl} with initial data $\phi$, constructed in Theorem \ref{mainresult} for all $\beta \geq 0$ in the time interval $[0,T]$ (remembering that $T$ is not dependent on $\beta$), then
	$$\lim_{\beta \to 0^{+}}\sup_{t\in [0,T]}\left\|u^{\beta}(t)-u^0(t)\right\|_{H^s}=0,$$
	where $u^0$ is the solution of equation \eqref{CCLequa1} with initial data $u^0(0)=\phi$.
\end{theorem} 
Then we study the convergence of the solutions of CL to solutions of the Benjamin-Ono ($\eta=0$) equation, when the dispersion parameter is fixed and the dissipation $\eta$ tends to zero. 
\begin{theorem}\label{convergeBO}
Let $\beta>0$, $\phi\in H^{s}(\mathbb{R})$, $s>\frac{3}{2}$ and let $u^{\eta}$ be the solution of CL satisfying $u^{\eta}(0)=\phi$. Then the limit $u^0=\lim_{\eta \to 0} u^{\eta}$ exists in $C\left([0,T];H^s(\mathbb{R})\right)\cap C^1\left([0,T];H^{s-2}(\mathbb{R})\right)$ and is the unique solution of the CL equation with $\beta=0$ that depends continuously on the initial data.
\end{theorem}
Finally, we state the result about decay properties of the solution of initial value problem CL which provide a theoretical prove of the numerical result posed in \cite{FeKa}.
\begin{theorem}\label{decaida}
Let $\eta >0$ be fixed and let $T>0$. Assume that $u\in C([0,T],\mathcal{F}_{3,3}(\mathbb{R}))$ is the solution of (\ref{cl}). Then $u(t)=0$, for all $t\in [0,T]$.
\end{theorem}

The layout of this paper is organized as follows: In Section $2$ we revisit the proof of the Theorem \ref{mainresult} given in \cite{PR}. Section $3$ is devoted to give a proof of Theorems \ref{malpuestodos} and \ref{malpuestoCLND} which say us that CL and CLND are ill-posed in $H^s(\mathbb{R})$ for $s<-1/2$. Section $4$ presents the study of the behavior of the solutions of the Cheen-Lee equation when the dispersive parameter $\beta$ tends to zero, we will give a proof of the Theorem \ref{convergeCLND}, and Section $5$ is dedicated to the study of the convergence of the solutions of the Cauchy problem CL to solutions of the Cauchy problem associated to the BO equation, we will give a proof of the Theorem \ref{convergeBO}. Finally, Section $6$ is devoted to study decay properties of the solution of initial value problem CL and a proof of the Theorem \ref{decaida}.
\setcounter{equation}{0}
\section{Theory in $H^s(\mathbb{R})$ with $s>-1/2$}

We recall that Theorem \ref{mainresult} was already proved by Pastr\'an in \cite{P} but here we present a different proof using the dissipative methods of Dix \cite{Dix} (see \cite{Amin, Duque, Pilod} for similar results). In fact, we revisit the proof of the Theorem \ref{mainresult} given in \cite{PR}. The main idea is to construct a contraction with the integral formulation of \eqref{cl}
\begin{equation}\label{intequation}
u(t)=S(t)\phi - \int_0^tS(t-t')[u(t')u_x(t')]\,dt' ,\quad t\geq 0,
\end{equation}
defined on an appropriated Banach space $\mathfrak{X}_T^s$, when $s>-\frac{1}{2}$ and $0<T\leq 1$. We introduce $\mathfrak{X}_T^s$ in order to deduce the crucial linear and bilinear estimates which are an adaptation, made by Esfahani \cite{Amin} and Duque \cite{Duque}, of the spaces originally presented by Dix in \cite{Dix} for the dissipative Burgers equation. For $s<0$ and $0\leq T\leq 1$, we define 
$$X_T^s=\left\{u\in C\left([0,T];H^s(\mathbb{R})\right)\, : \, \left\|u\right\|_{X_T^s} < \infty \right\},$$ 
where
\begin{equation}
\left\|u\right\|_{X_T^s}:=\sup_{t\in(0,T]}\left(\left\|u(t)\right\|_{H^s}+t^{|s|/2}\left\|u(t)\right\|_{L^2}\right).
\end{equation}

We start giving the following technical results.
\begin{prop} \label{semigroupregular}
Let $\lambda \geq 0$ and $s\in \mathbb{R}$. Then, \\ \\
(a.) $S(t)\in \textbf{B}(H^s(\mathbb{R}),H^{s+\lambda}(\mathbb{R}))$ for all $t>0$ and satisfies,
\begin{align}
\nor{S(t)\phi}{s+\lambda}\leq C_{\lambda}\,\bigr(e^{\eta t}+(\eta t)^{-\lambda /2}\bigl)\,\nor{\phi}{s}\;,  \label{regulariza}
\end{align}
where $\phi \in H^{s}(\mathbb{R})$ and $C_{\lambda}$ is a constant depending only on $\lambda$. Moreover, the map $t\to S(t)\phi$ belongs to $C((0,\infty),H^{s+\lambda}(\mathbb{R}))$.\\ \\
(b.) $S:[0,\infty)\longrightarrow \textbf{B}(H^s(\mathbb{R}))$ is a $C^0$-semigroup in $H^s(\mathbb{R})$. Moreover, for every $t\geq 0$,
\begin{equation}\label{linealsemigroupcota}
\nor{S(t)}{\textbf{B}(H^s)}\leq e^{\eta t}.
\end{equation}
\end{prop}
\begin{lemma}\label{LemmaLB1}
Let $\lambda>0$, $\eta>0$ and $t>0$ be given. Then
$$\left\||t\xi^2|^{\lambda} e^{\eta\left(|\xi|-\xi^2 \right)t}\right\|_{L^{\infty}(\mathbb{R})}\lesssim_{\lambda} f_{\lambda}(t),$$
where
\begin{equation}\label{LBequa1}
f_{\lambda}(t)=\left(t^{\lambda}+\eta^{-\lambda}\right) e^{\frac{\eta}{8}\left(t+t^{\frac{1}{2}}\sqrt{t+\frac{16\lambda}{\eta}}\right)},
\end{equation}
is a nondecreasing function defined for all $t>0$.
\end{lemma}

\begin{proof}
For all $\xi \in \mathbb{R}$ we have that
$$|t\xi^2|^{\lambda} e^{\eta\left(|\xi|-\xi^2 \right)t}\leq  \sup_{x\in \mathbb{R}} |x|^{2\lambda} e^{\eta\left(|x|t^{1/2}-x^2 \right)}.$$
Let $w_t(x)= x^{2\lambda} e^{\eta\left(xt^{1/2}-x^2 \right)}$, for all $x\geq 0$. Note that $w_t(x)$ tends to $0$ as $x\to \infty$, and  
$$w_t'(x_1)=0 \, \Longleftrightarrow \,  x_{1}=\frac{1}{4}\left(t^{\frac{1}{2}}+\sqrt{t+\frac{16\lambda}{\eta}}\right).$$
Therefore, the maximum of $w_t$ is attained in $x_1$ and we can deduce that
$$w_t(x_1)\lesssim_{\lambda} \left(t^{\lambda}+\eta^{-\lambda}\right) e^{\frac{\eta}{8}\left(t+t^{\frac{1}{2}}\sqrt{t+\frac{16\lambda}{\eta}}\right)}.$$
This inequality completes the proof.
\end{proof}
From Lemma \ref{LemmaLB1} and arguing as in Proposition 1 in \cite{Pilod}, it is easy to deduce the next proposition.

\begin{prop}\label{PropLB1}
Let $\eta>0$, $s \in \mathbb{R}$, $T>0$ and $\phi\in H^s(\mathbb{R})$, then it follows that
\begin{equation}
\sup_{t\in[0,T]}\left\|S(t)\phi\right\|_{H^s}\leq e^{\frac{\eta}{4}T} \left\|\phi\right\|_{H^s}.
\end{equation}
And when $0< T\leq 1$ and $s< 0$,
\begin{equation}
\sup_{t\in[0,T]}t^{\frac{|s|}{2}}\left\|S(t)\phi\right\|_{L^2}\lesssim_{s}g_{s,\eta}(T)\left\|\phi\right\|_{H^s},
\end{equation}
where $$g_{s,\eta}(t)=e^{\frac{\eta t}{4}}+ \left(t^{\frac{|s|}{2}}+\eta^{-\frac{|s|}{2}}\right) e^{\frac{\eta}{8}\left(t+t^{\frac{1}{2}}\sqrt{t+\frac{8|s|}{\eta}}\right)},$$
is a nondecreasing function on $[0,1]$.
\end{prop}

Next, we establish the crucial bilinear estimates. 

\begin{prop}\label{PropLB2} Let $0\leq T\leq 1$ and $-\frac{1}{2}<s < 0$, then
\begin{equation}
\left\|\int_{0}^t S(t-t')\partial_x(uv)(t') \ dt'\right\|_{X_T^s} \lesssim_{s} e^{\frac{\eta T}{2}} T^{\frac{1+2s}{4}}\left\|u\right\|_{X_T^s}\left\|v\right\|_{X_T^s},
\end{equation}
for all $u,v\in X_T^s$.
\end{prop}

\begin{proof}
Since $s< 0$, it follows that $\left\langle\xi \right\rangle^{s}\leq |\xi|^s$, for all real number $\xi$ different from zero. Then we deduce that
\begin{equation}\label{LBequa2}
\begin{aligned}
&  \left\|\int_{0}^t S(t-t')\partial_x(uv)(t') \ dt'\right\|_{H^s} \\
& \hspace{30pt} \leq  \int_{0}^t \left\|\left\langle \xi \right\rangle^s e^{\eta\left(|\xi|-\xi^2\right)(t-t')} \left(\partial_x(uv)(t')\right)^{\wedge}(\xi) \right\|_{L^2(\mathbb{R})} \ dt' \\
& \hspace{30pt} \leq  \int_{0}^t \left\||\xi|^{1+s}e^{\eta\left(|\xi|-\xi^2\right)(t-t')}\right\|_{L^2(\mathbb{R})}\left\| \widehat{u(t')}\ast \widehat{v(t')}(\xi) \right\|_{L^{\infty}(\mathbb{R})} \ dt'. 
\end{aligned}
\end{equation}
The Young inequality implies that
\begin{equation}\label{LBequa3}
\left\| \widehat{u(t')}\ast\widehat{v(t')}(\xi) \right\|_{L^{\infty}(\mathbb{R})}\leq \left( \frac{\left\|u\right\|_{X_T^s}\left\|v\right\|_{X_T^s}}{|t'|^{|s|}} \right),
\end{equation}
thus we obtain
\begin{equation}\label{LBequa4}
\begin{aligned}
\int_{0}^t & \left\| S(t-t')\partial_x(uv)(t') \right\|_{H^s}  \ dt'  \\
& \qquad \leq  \int_{0}^t \frac{\left\||\xi|^{1+s}e^{\eta\left(|\xi|-\xi^2\right)t}\right\|_{L^2(\mathbb{R})}}{|t-t'|^{|s|}} \, dt' \, \left\|u\right\|_{X_T^s}\left\|v\right\|_{X_T^s}. 
\end{aligned}
\end{equation}
To estimate the integral on the right-hand side of \eqref{LBequa4}, we perform the change of variables $w=t^{1/2}\xi$ to deduce 
\begin{align}\label{LBequa5}
\left\||\xi|^{1+s}e^{\eta\left(|\xi|-\xi^2\right)t}\right\|_{L^2(\mathbb{R})}& \leq \frac{\left\||w|^{1+s}e^{-\frac{\eta w^2}{2}}\right\|_{L^2(\mathbb{R})}\left\|e^{\eta(wt^{1/2}- \frac{w^2}{2})}\right\|_{L^{\infty}(\mathbb{R})}}{t^{\frac{3}{4}+\frac{s}{2}}} \nonumber \\
& \lesssim_{s} \frac{e^{\frac{\eta T}{2}}}{t^{\frac{3}{4}+\frac{s}{2}}}.
\end{align}
Therefore, we get from \eqref{LBequa4} and \eqref{LBequa5} that
\begin{equation}\label{LBequa6}
\begin{aligned}
& \left\|\int_{0}^t  S(t-t')\partial_x(uv)(t')  \ dt' \right\|_{H^s}  \\
& \hspace{20pt} \lesssim_s e^{\frac{\eta T}{2}}t^{\frac{1}{4}(1+2s)}\left( \int_{0}^1 \frac{1}{\sigma^{\frac{3}{4}+\frac{s}{2}}|1-\sigma|^{|s|}}\ d\sigma \right)\left\|u\right\|_{X_T^s}\left\|v\right\|_{X_T^s},   
\end{aligned}
\end{equation}
for all $0\leq t \le T$. On the other hand, arguing as above, we have for all $0\leq t\leq T$ that
\begin{align}\label{LBequa7}
 t^{|s|/2}&\left\|\int_{0}^t S(t-t')\partial_x(uv)(t') \ dt'\right\|_{L^2(\mathbb{R})} \nonumber \\
&\leq t^{|s|/2} \int_{0}^t \left\||\xi|e^{\eta\left(|\xi|-\xi^2\right)(t-t')}\right\|_{L^2(\mathbb{R})}\left\| \widehat{u(t')}\ast \widehat{v(t')}(\xi) \right\|_{L^{\infty}(\mathbb{R})} \ dt' \nonumber \\
&  \leq t^{|s|/2}\int_{0}^t \frac{\left\||\xi|e^{\eta\left(|\xi|-\xi^2\right)t}\right\|_{L^2(\mathbb{R})}}{|t-t'|^{|s|}} \, dt' \, \left\|u\right\|_{X_T^s}\left\|v\right\|_{X_T^s}  \nonumber \\
& \lesssim_s e^{\frac{\eta T}{2}} T^{\frac{1}{4}(1+2s)} \left( \int_{0}^1 \sigma^{-\frac{3}{4}}|1-\sigma|^{s}\, d\sigma \right)\left\|u\right\|_{X_T^s}\left\|v\right\|_{X_T^s}.
\end{align}
Combing \eqref{LBequa6} and \eqref{LBequa7} the proof is complete.
\end{proof}

\begin{rem}
If we consider $s'>s>-\frac{1}{2}$. Then modifying the space $X_{T}^{s'}$ by
$$\tilde{X}_{T}^{s'}=\left\{u\in X_{T}^{s'}: \left\|u\right\|_{\tilde{X}_{T}^{s'}}<\infty \right\},$$
where 
$$ \left\|u\right\|_{\tilde{X}_{T}^{s'}}= \left\|u\right\|_{X_{T}^{s'}}+t^{|s|/2} \left\|(1-\partial_x^2)^{\frac{ s'-s}{2}} u\right\|_{L^2}$$
and using that
$$(1+\xi^2)^{s/2}\lesssim (1+\xi^2)^{s/2}(1+\xi_1^2)^{(s'-s)/2}+(1+\xi^2)^{s/2}\left(1+(\xi-\xi_1)^2\right)^{(s'-s)/2},$$
for all $\xi,\xi_1\in \mathbb{R}$, we deduce arguing as in Proposition \eqref{PropLB2} that
$$\left\|\int_{0}^t S(t-t')\partial_x(uv)(t') \ dt'\right\|_{\tilde{X}_T^{s'}} \lesssim_{s}e^{\frac{\eta T}{2}} T^{\frac{1+2s}{4}}\left(\left\|u\right\|_{\tilde{X}_T^{s'}}\left\|v\right\|_{X_T^s}+\left\|u\right\|_{X_T^s}\left\|v\right\|_{\tilde{X}_T^{s'}}\right).$$
\end{rem}

\begin{prop}\label{PropLB3}
Let $0\leq T \leq 1$, $s \in (-\frac{1}{2},0)$ and $\delta\in [0,s+\frac{1}{2})$, then the application 
$$t\rightarrow \int_{0}^t S(t-t')\partial_x(u^2)(t')\ dt' ,$$
is in $C\left((0,T];H^{s+\delta}(\mathbb{R})\right)$, for every $u\in X_T^s$.
\end{prop}
\begin{proof}
The proof of this proposition is similar to the proof of Proposition 4 in \cite{Pilod}. 
\end{proof}

The next lemma will enable us to estimate the term $\partial_x(uv)$ when $u,v\in C\left([0,T];H^s(\mathbb{R})\right)$, with $s\geq 0$.

\begin{lemma}\label{LemmaLB2}
Suppose $a>0$, $r\geq 0$ are real numbers and $\phi,\psi \in H^r(\mathbb{R})$. Then
\begin{equation}
\left\|\left\langle a\xi \right\rangle^r(\phi \psi)^{\wedge}(\xi)\right\|_{L^{\infty}(\mathbb{R})}\leq 2^{\frac{r}{2}} \left\|\left\langle a\xi \right\rangle^{r}\widehat{\phi}(\xi)\right\|_{L^2(\mathbb{R})}\left\|\left\langle a\xi \right\rangle^{r}\widehat{\psi}(\xi)\right\|_{L^2(\mathbb{R})}.
\end{equation}
\end{lemma}

\begin{proof}
See proof of Lemma 2.3.1 in Dix \cite{Dix}.
\end{proof}

\begin{rem}\label{RemarkLB1}
Assuming that $s\geq 0$ and $0<T\leq 1$, we have a similar result as the one obtained in Proposition \ref{PropLB2} for the space $C\left([0,T];H^s(\mathbb{R})\right)$. In fact, we have that
$$\left\|\int_{0}^t S(t-t')\partial_x(uv)(t') \ dt'\right\|_{L^{\infty}_tH^s_x} \lesssim_{s} e^{\frac{\eta T}{2}} T^{\frac{1}{4}}\left\|u\right\|_{L^{\infty}_tH^s_x}\left\|v\right\|_{L^{\infty}_tH^s_x},$$
for all $u,v\in C\left([0,T];H^s(\mathbb{R})\right).$
To deduce this result, from Lemma \ref{LemmaLB2} with $a=1$ and inequality \eqref{LBequa5} we get
\begin{align} 
 \int_{0}^{t}  &\left\|S(t-t')\partial_x(uv)(t')\right\|_{H^s} \, dt'  \nonumber \\
 &\lesssim \int_{0}^t \left\| |\xi| e^{\eta\left(|\xi|-\xi^2\right)(t-t')} \right\|_{L^2(\mathbb{R})} \left\|\left\langle \xi \right\rangle^s \left(uv(t')\right)^{\wedge}(\xi)\right\|_{L^{\infty}(\mathbb{R})}  \, dt' \nonumber \\
 & \lesssim_s \int_{0}^{t} \left\||\xi| e^{\eta\left(|\xi|-\xi^2\right)(t-t')}\right\|_{L^2(\mathbb{R})}\left\|u(t')\right\|_{H_x^s}\left\|v(t')\right\|_{H^s_x}\ dt'  \nonumber \\
 & \lesssim_{s} e^{\frac{\eta T}{2}} T^{\frac{1}{4}}\left\|u\right\|_{L^{\infty}_tH^s_x}\left\|v\right\|_{L^{\infty}_tH^s_x}.
\end{align}
\end{rem}

\begin{rem}\label{RemarkLB2}
Let $s\geq 0$ and $0<T\leq 1$. We have the same result given in Proposition \ref{PropLB3}, changing $X_{T}^s$ by $C\left([0,T];H^s(\mathbb{R})\right)$ and taking $\delta\in[0,\frac{1}{2})$. This result is proved using Lemma \ref{LemmaLB2} and arguing as in the proof of Proposition 4 in \cite{Pilod}.
\end{rem}


\subsection{Well-posedness in $H^s(\mathbb{R})$, $s>-1/2$.}

\begin{proof}[Proof of Theorem \ref{mainresult}] For a fixed $T\in(0,1]$, we take $\mathfrak{X}_T^s=X_T^s$, if $-\frac{1}{2}< s < 0$, and when $s> 0$, we make $\mathfrak{X}_T^s=C\left([0,T];H^s(\mathbb{R})\right)$. We divide the proof in five steps
\\ \\
1. \emph{Existence}. Let $\phi \in H^s(\mathbb{R})$ with $s> -\frac{1}{2}$. We consider the application
$$\Psi(u)=S(t)\phi-\frac{1}{2}\int_{0}^t S(t-t')\partial_x(u^2(t')) \ dt',$$
for each $u \in \mathfrak{X}_T^s$. By Proposition \ref{PropLB1}, together with Proposition \ref{PropLB2} when $s< 0$, or by Remark \ref{RemarkLB1} when $s\geq 0$, there exists a positive constant $C=C(\eta,s)$ such that 
\begin{align}
\left\|\Psi(u)\right\|_{\mathfrak{X}_T^s} &\leq C\left(\left\|\phi\right\|_{H^s}+T^{g(s)}\left\|u\right\|_{\mathfrak{X}_T^s}^2\right), \label{WPequa1} \\
\left\|\Psi(u)-\Psi(v)\right\|_{\mathfrak{X}_T^s} & \leq C T^{g(s)}\left\|u-v\right\|_{\mathfrak{X}_T^s}\left\|u+v\right\|_{\mathfrak{X}_T^s}, \label{WPequa2}
\end{align}
for all $u, v \in \mathfrak{X}_T^s$ and $0<T\leq 1$. Where $g(s)=\frac{1}{4}(1+2s)$, when $s\in (-\frac{1}{2},0)$ and $g(s)=\frac{1}{4}$, for all $s\geq 0$. Then, we define $E_{T}(\gamma)=\left\{u\in \mathfrak{X}_T^s : \left\|u \right\|_{\mathfrak{X}_T^s}\leq\gamma \right\}$, with $\gamma=2C\left\|\phi \right\|_{H^s}$ and $0<T\leq \min \left\{1,\left(4C\gamma\right)^{-\frac{1}{g(s)}} \right\}$. The estimates \eqref{WPequa1} and \eqref{WPequa2} imply that $\Psi$ is a contraction on the complete metric space $E_T(\gamma)$.  Therefore, the Fixed Point Theorem implies the existence of a unique solution $u$ of \eqref{intequation} in $E_T(\gamma)$ with $u(0)=\phi$. 
\\ \\
2. \emph{Continuous dependence}. We will verify that the map $\phi \in H^s(\mathbb{R}) \mapsto u \in \mathfrak{X}_T^s$, where $u$ is a solution of \eqref{cl} obtained in the step of \emph{Existence} is continuous. More precisely, for $s>-\frac{1}{2}$, if $\phi_n \rightarrow \phi_{\infty}$ in $H^s(\mathbb{R})$, let $u_n\in \mathfrak{X}_{T_n}^s$ be the respective solutions of \eqref{intequation} (obtained in the part of  \emph{Existence}) with $u_n(0)=\phi_n$, for all $1\leq n\leq \infty$. Then for each $T'\in (0,T_{\infty})$, $u_n \in \mathfrak{X}_{T'}^s$ (for $n$ large enough) and $u_n \rightarrow u_{\infty}$ in $\mathfrak{X}_{T'}^s$.

We recall that the solutions and times of existence previously constructed satisfy 
\begin{align}
&0<T_n\leq \min\left\{1, \left(8C^2\left\|\phi_n\right\|_{H^s}\right)^{-\frac{1}{g(s)}}\right\}, \label{WPequa3} \\
&\left\|u_n\right\|_{\mathfrak{X}_T^s} \leq 2C\left\|\phi_n\right\|_{H^s}, \label{WPequa4}
\end{align}
for all $n\in \mathbb{N}\cup\left\{\infty\right\}$. Let $T'\in (0,T_{\infty})$, the above inequalities and the hypothesis imply that there exists $N\in \mathbb{N}$, such that for all $n\geq N$, we have that $T'\leq T_n$ and
$$\frac{\left\|\phi_n\right\|_{H^s}+\left\|\phi_{\infty}\right\|_{H^s}}{\left\|\phi_{\infty}\right\|_{H^s}}\leq 3.$$
Therefore, combining \eqref{WPequa3}, \eqref{WPequa4} with the Propositions \ref{PropLB1}, \ref{PropLB2} when the index $s$ is negative, or with the Remark \ref{RemarkLB1} when $s\geq 0$, it follows that for each $n\geq N$ 
\begin{align*}
\left\|u_n-u_{\infty}\right\|_{\mathfrak{X}_{T'}^s} &\leq C\left\|\phi_n-\phi_{\infty}\right\|_{H^s}+ CT_{\infty}^{g(s)}\left\|u_n+u_{\infty}\right\|_{\mathfrak{X}_{T'}^s}\left\|u_n-u_{\infty}\right\|_{\mathfrak{X}_{T'}^s} \\
&\leq C\left\|\phi_n-\phi_{\infty}\right\|_{H^s}+ \frac{\left(\left\|\phi_n\right\|_{H^s}+\left\|\phi_{\infty}\right\|_{H^s}\right)}{4\left\|\phi_{\infty}\right\|_{H^s}}\left\|u_n-u_{\infty}\right\|_{\mathfrak{X}_{T'}^s} \\
&\leq C\left\|\phi_n-\phi_{\infty}\right\|_{H^s}+ \frac{3}{4}\left\|u_n-u_{\infty}\right\|_{\mathfrak{X}_{T'}^s}. 
\end{align*}
Hence we have deduced that $\left\|u_n-u_{\infty}\right\|_{\mathfrak{X}_{T'}^s}\leq C\left\|\phi_n-\phi_{\infty}\right\|_{H^s}$, for all $n\geq N$.
\\ \\
3. \emph{Uniqueness}. Let $u,v \in \mathfrak{X}_T^s$ be solutions of the integral equation \eqref{intequation} on $[0,T]$ with the same initial data. For each $r\in [0,T]$ we define
$$
F_{r}(t)=
\begin{cases}
\frac{1}{2}\int_{r}^{t} S(t-t')\left(\partial_x u^2(t')-\partial_x v^2(t')\right)dt', & \text{si }t\in (r,T] \\
0, & \text{si }t\in[0,r]
\end{cases}
$$
for all $t \in [0,T]$. Arguing as in the proof of Proposition \ref{PropLB2} or Remark \ref{RemarkLB1}, we deduce that there exists a positive constant $C=C(\eta,s)$ depending only on $\eta$ and $s$, such that for all $r\in[0,T]$ and all $\vartheta\in [r,T]$,
\begin{equation}\label{WPequa5}
\left\|F_{r}\right\|_{\mathfrak{X}_{\vartheta}^s} \leq C K \left(\vartheta-r\right)^{g(s)}\left\|u-v\right\|_{\mathfrak{X}_{\vartheta}^s},
\end{equation}
where $K=\left\|u\right\|_{\mathfrak{X}_{T}^s}+\left\|v\right\|_{\mathfrak{X}_{T}^s}$. In particular, inequality \eqref{WPequa5} implies that
\begin{equation}\label{WPequa6}
\left\|u-v\right\|_{\mathfrak{X}_{\vartheta}^s}=\left\|F_{0}\right\|_{\mathfrak{X}_{\vartheta}^s} \leq C K \vartheta^{g(s)}\left\|u-v\right\|_{\mathfrak{X}_{\vartheta}^s}.
\end{equation}
Thus, choosing $\vartheta \in \left(0,(CK)^{-\frac{1}{g(s)}}\right)$ a fixed number, \eqref{WPequa6} implies that $u \equiv v$ on $[0,\vartheta]$. Therefore we can iterate this argument using \eqref{WPequa5} and our choose of $\vartheta$, until we extend the uniqueness result to the whole interval $[0,T]$.
\\ \\
 \emph{4. The solution $ u\in C\left((0,T],H^{\infty}(\mathbb{R})\right)$}. From Lemma \ref{LemmaLB1} and arguing as in the proof of Proposition 2.2 in \cite{BI}, we have that the map $t\mapsto S(t)\phi$ is continuous in the interval $(0,T]$ with respect to the topology of $H^{\infty}(\mathbb{R})$. Since our solution $u$ is in $\mathfrak{X}_T^s$, we deduce from Proposition \ref{PropLB3} or Remark \ref{RemarkLB2}, that there exists $\lambda>0$, such that 
$$u\in C\left([0,T];H^s(\mathbb{R})\right)\cap C\left((0,T];H^{s+\lambda}(\mathbb{R})\right).$$ 
Therefore we can iterate this argument, using uniqueness result and the fact that the time of existence of solutions depends uniquely of the $H^s(\mathbb{R})$-norm of the initial data. Thus we deduce that 
$$u\in C\left([0,T];H^s(\mathbb{R})\right)\cap C\left((0,T];H^{\infty}(\mathbb{R})\right).$$
\\ \\
5. \emph{Global well-posedness}. Since Pastr\'an proved in \cite{P} that CL is globally well posed for all $\phi\in H^s(\mathbb{R})$ when $s\geq 0$, we shall prove that CL is globally well posed in $X_{T}^s$ when $-\frac{1}{2}<s<0$. In fact, let $s\in (-1/2,0)$, $\phi\in H^s(\mathbb{R})$ and $u\in X_{T}^s$ be the solution of the Cauchy problem \eqref{cl} obtained in above steps. Let $T'\in (0,T)$ fixed, we have that
$$
\left\|u\right\|_{X_{T'}^s}=M_{T',s}<\infty.
$$
Since $u\in C\left((0,T];H^{\infty}(\mathbb{R})\right)$, it follows that $u(T')\in L^2(\mathbb{R})$. Thus, Theorem 1.2 in \cite{P} implies that $\tilde{u}$, the solution of \ref{intequation} with initial data $u(T')$, is global in time. Moreover, uniqueness implies that $\tilde{u}(t)=u(T'+t)$ for all $t\in [0,T-T']$. Therefore, we deduce that
\begin{align*}
\left\|u\right\|_{X_{T}^s} & \leq \left\|u\right\|_{X_{T'}^s}+\left\|u(T'+\cdot)\right\|_{X_{T-T'}^s} \\
&\leq M_{T',s}+\left\|\tilde{u}\right\|_{X_{T-T'}^s} \\
&= M_{T',s}+ \sup_{t\in [0,T-T']} \left\{\left\|\tilde{u}(t)\right\|_{H^s}+t^{|s|/2}\left\|\tilde{u}(t)\right\|_{L^2(\mathbb{R})}\right\} \\
& \leq M_{T',s}+ \left(1+(T-T') ^{|s|/2}\right)\sup_{t\in [0,T-T']}\left\|\tilde{u}(t)\right\|_{L^2(\mathbb{R})}. 
\end{align*}
The global result follows from the above estimate.
\end{proof}

\section{ Ill-posedness results}

Without using the Fourier restriction norm method we have deduced for the Cauchy problem \eqref{cl} local and global well-posedness in $H^s(\mathbb{R})$, when $s>-1/2$. We shall prove that this result is sharp in the sense that the flow-map data-solution fails to be $C^3$ in $H^s(\mathbb{R})$ for $s<-\frac{1}{2}$. We recall that Pastr\'an proved in \cite{P} that the assumption of $C^2$ regularity in $H^s(\mathbb{R})$ for the flow-map of the Cheen-Lee equation fails when $s<-1$.

\begin{proof}[Proof of Theorem \ref{malpuestodos}]
Let $s<-\frac{1}{2}$, suppose that there exists $T>0$ such that the Cauchy problem \eqref{cl} is locally well-posed in $H^s(\mathbb{R})$ on the time interval $[0,T]$ and such that the flow map data-solution $\Phi(t): H^s(\mathbb{R})\longrightarrow C\left([0,T];H^s(\mathbb{R})\right), \, \phi \longmapsto u\left(t\right)$ is $C^3$ at the origin. When $\phi \in H^s(\mathbb{R})$, we have that $\Phi(\cdot)\phi$ is a solution of the equation \eqref{cl} with initial data $\phi$. This means that $\Phi(\cdot)\phi$ is a solution of the integral equation

$$\Phi(t)\phi=S(t)\phi-\frac{1}{2}\int_{0}^t S(t-t')\partial_x(\Phi(t)\phi)^2dt'.$$

Since the Cauchy problem \eqref{cl} is supposed to be well-posed, we know using the uniqueness that $\Phi(t)(0)=0$. Thus, with this result and using the integral equation, we see that
\begin{align*}
&u_1(t)= d_{0}\Phi(t)(\phi)=S(t)\phi, \\
&u_2(t)=d^2_0\Phi(t)(\phi,\phi) \sim \int_0^t S(t-t')\partial_x\left(u_1(t')u_1(t')\right)\, dt', \\
&u_3(t)=d^3_0\Phi(t)(\phi,\phi,\phi)\sim \int_0^t S(t-t')\partial_x \left(u_1(t)u_2(t)\right)\, dt'. 
\end{align*}
The assumption of $C^3$ regularity implies that $d^3_0\Phi(t)\in \mathcal{B}\left(H^s(\mathbb{R})\times H^s(\mathbb{R})\times H^s(\mathbb{R}),H^s(\mathbb{R})\right)$, which would lead to the following inequality
\begin{equation}\label{IPequa1}
\left\|u_3(t)\right\|_{H^s}\lesssim \left\|\phi\right\|_{H^s}^3, \forall \phi \in H^s(\mathbb{R}).
\end{equation}
We will show that \eqref{IPequa1} fails for an appropriated function $\phi.$ We define $\phi$ by his Fourier transform as
\begin{equation}\label{IPequa2}
\widehat{\phi}(\xi)=N^{-s}\gamma^{-\frac{1}{2}}\left(\chi_{I_N}(\xi)+\chi_{I_N}(-\xi)\right),
\end{equation}
where $I_N=[N,N+2\gamma],$ $N\gg 1$ and $\gamma =\epsilon N$, with $0<\epsilon\ll 1$ fixed. We note that
$$\left\|\phi\right\|_{H^s}\sim_s 1.$$
On the other hand, we have that
\begin{equation}\label{IPequa3}
\widehat{u_3(t)}(\xi)=i\xi \int_{0}^t e^{iq(\xi)(t-t')-p(\xi)(t-t')} \widehat{u_1(t')} \ast \widehat{u_2(t')}(\xi) \,  dt'.
\end{equation}
Then, using the definition of the group $S(t)$ and Fubini's Theorem, we get
\begin{align}\label{IPequa4}
\widehat{u_2(t)}(\xi)&=i\int_{0}^t\xi e^{iq(\xi)(t-t')-p(\xi)(t-t')}\left\{e^{iq(\cdot)t'-p(\cdot)t'}\widehat{\phi}\ast e^{iq(\cdot)t'-p(\cdot)t'}\widehat{\phi} \right\}(\xi) \, dt' \nonumber\\
&= i\xi e^{iq(\xi)t-p(\xi)t}\int_{\mathbb{R}}\int_{0}^t\widehat{\phi}\left(\xi-\xi_1\right)\widehat{\phi}\left(\xi_1\right)e^{i\left(q(\xi_1)+q(\xi-\xi_1)-q(\xi)\right)t'-\left(p(\xi_1)+p(\xi-\xi_1)-p(\xi)\right)t'} \, dt' \, d\xi_1 \nonumber \\
&= i\xi e^{iq(\xi)t-p(\xi)t}\int_{\mathbb{R}}\widehat{\phi}\left(\xi-\xi_1\right)\widehat{\phi}\left(\xi_1\right)\frac{e^{\sigma(\xi,\xi_1)t}-1}{\sigma(\xi,\xi_1)} \, d\xi_1,
\end{align}
where
\begin{align*}
\sigma(\xi,\xi_1)&=i\left(q(\xi_1)+q(\xi-\xi_1)-q(\xi)\right)-\left(p(\xi_1)+p(\xi-\xi_1)-p(\xi)\right) \\
&= i\beta\left(|\xi-\xi_1|(\xi-\xi_1)-|\xi|\xi+|\xi_1|\xi_1\right)-\eta\left((\xi-\xi_1)^2-|\xi-\xi_1|-\xi^2+|\xi|+|\xi_1|^2-|\xi_1|\right).
\end{align*}
Therefore, by \eqref{IPequa4} and Fubini's theorem we have 
\begin{align*}
&\widehat{u_3(t)}(\xi)=c\xi e^{iq(\xi)t-p(\xi)t}\int_{0}^t \int_{\mathbb{R}} e^{-iq(\xi)t'+p(\xi)t'} \widehat{u_2(t')}(\xi_2)\widehat{u_1}(\xi-\xi_2) \, d\xi_2 \, dt' \\
& =c\xi e^{iq(\xi)t-p(\xi)t}\int_{0}^t \int_{\mathbb{R}^2} \widehat{\phi}(\xi_1)\widehat{\phi}(\xi_2-\xi_1)\widehat{\phi}(\xi-\xi_2) \xi_2 e^{\sigma(\xi,\xi_2)t'}\left( \frac{e^{\sigma(\xi_2,\xi_1)t'}-1}{\sigma(\xi_2,\xi_1)}\right) \, d\xi_1 \, d\xi_2 \, dt' \\
& =c\xi e^{iq(\xi)t-p(\xi)t} \int_{\mathbb{R}^2} \widehat{\phi}(\xi_1)\widehat{\phi}(\xi_2-\xi_1)\widehat{\phi}(\xi-\xi_2)\frac{ \xi_2}{\sigma(\xi_2,\xi_1)} \left(\frac{e^{\psi(\xi,\xi_1,\xi_2)t}-1}{\psi(\xi,\xi_1,\xi_2)} -\frac{e^{\sigma(\xi,\xi_2)t}-1}{\sigma(\xi,\xi_2)}\right) \, d\xi_1 \, d\xi_2, \\
\end{align*}
where we set
$$\psi(\xi,\xi_1,\xi_2)=\sigma(\xi,\xi_2)+\sigma(\xi_2,\xi_1).$$

By support considerations, we observe that for all $\xi\in [N+3\gamma,N+4\gamma]$ 
\begin{equation}\label{IPequa5}
\left|\widehat{u_3(t)}(\xi)\right|\geq c N^{-3s}\gamma^{-\frac{3}{2}}\left|\xi e^{iq(\xi)t-p(\xi)t} \int_{K_{\xi}} \frac{ \xi_2}{\sigma(\xi_2,\xi_1)} \left(\frac{e^{\psi(\xi,\xi_1,\xi_2)t}-1}{\psi(\xi,\xi_1,\xi_2)} -\frac{e^{\sigma(\xi,\xi_2)t}-1}{\sigma(\xi,\xi_2)}\right) \, d\xi_1 \, d\xi_2\right|,
\end{equation}
where $K_{\xi}=K_{\xi}^1\cup K_{\xi}^2 \cup K_{\xi}^3$ and
\begin{align*}
K_{\xi}^1=\left\{(\xi_1,\xi_2):\xi_1\in I_N, \, \xi_2-\xi_1\in I_N, \, \xi-\xi_2\in -I_N \right\}, \\
K_{\xi}^2=\left\{(\xi_1,\xi_2):\xi_1\in I_N, \, \xi_2-\xi_1\in -I_N, \, \xi-\xi_2\in I_N \right\}, \\
K_{\xi}^3=\left\{(\xi_1,\xi_2):\xi_1\in -I_N, \, \xi_2-\xi_1\in I_N, \, \xi-\xi_2\in I_N \right\}. \\
\end{align*}
When $\xi\in [N+3\gamma,N+4\gamma]$ and $(\xi_1,\xi_2)\in K_{\xi}$, our choice of $\gamma$ implies
\begin{align*}
\left|\sigma\left(\xi_2,\xi_1\right)\right|\sim \left(\beta+\eta\right)N^2, \\
\left|\psi\left(\xi,\xi_1,\xi_2\right)\right|\sim\left|\sigma\left(\xi,\xi_2\right)\right|\sim \left(\beta+\eta\right)N^2.
\end{align*}
Hence we obtain that
\begin{equation}\label{IPequa6}
\left|\frac{1}{\sigma(\xi_2,\xi_1)}\left(\frac{e^{\psi(\xi,\xi_1,\xi_2)t_N}-1}{\psi(\xi,\xi_1,\xi_2)} -\frac{e^{\sigma(\xi,\xi_2)t_N}-1}{\sigma(\xi,\xi_2)}\right)\right|=t_N^2+O\left(t_N^3N^4\right).
\end{equation}
Now, setting a time $t_N:=N^{-2-\epsilon},$ we have that $e^{\eta\left(|\xi|-\xi^2\right)t_N}\sim e^{-\eta N^2t_N}\sim e^{-\eta N^{-\epsilon}}=C>0$, for all $\xi \in [N+3\gamma,N+4\gamma]$. Thus, since the main contribution in \eqref{IPequa6} is given by $N^{-4-2\epsilon}$, using that $|\xi_2| \sim N$ and $mes(K_{\xi})\gtrsim \gamma^2$, we deduce from \eqref{IPequa5} that 
\begin{align*}
\left|\widehat{u_3(t)}(\xi)\right|& \chi_{[N+3\gamma,N+4\gamma]}(\xi) \\
 &\gtrsim_{\eta,\beta} N^{-3s-2-2\epsilon}\gamma^{\frac{1}{2}}\chi_{[N+3\gamma,N+4\gamma]}(\xi).
\end{align*}
Therefore, we conclude
$$\left\|u_{3}(t_N)\right\|_{H^s}\gtrsim_{\eta,\beta,s} \gamma N^{-2s-2-2\epsilon}\sim N^{-2s-1-2\epsilon},$$
which contradicts that $\left\|\phi\right\|_{H^s} \sim_{s} 1$ for $N$ large enough, since $s<-\frac{1}{2}$.
\end{proof}
\begin{proof}[Proof of Theorem \ref{malpuestoCLND}]
Let $\phi \in H^s(\mathbb{R})$ and  $u_1(t)$, $u_2(t)$ as in the proof of Theorem \ref{malpuestodos}. The assumption of $C^2$ regularity implies that 
\begin{equation}\label{IPequa7}
\left\|u_2(t)\right\|_{H^s} \leq C\left\|\phi\right\|_{H^s}^2.  
\end{equation}
We will proof that \eqref{IPequa7} fails for an appropriated function $\phi$. Let $\phi$ defined as in \eqref{IPequa2}, but in this case we consider $\gamma=N^{1-\epsilon}$ with $0<\epsilon \ll 1$. By support considerations we have that for all $\xi \in [2N,2N+4\gamma]$
\begin{align}
\left|\widehat{u_2(t)}(\xi)\right| \geq \left|c N^{-2s}\gamma^{-1} \xi e^{\eta(|\xi|-\xi^2)t}\, \int_{K_{\xi}}\frac{e^{\lambda(\xi,\xi_1)t}-1}{\lambda(\xi,\xi_1)} \, d\xi_1 \right|,
\end{align}
where
\begin{align*}
&K_{\xi}=\left\{\xi_1 \, : \, \xi-\xi_1\in I_{N}, \, \xi_1\in I_N \right\}, \\
&\lambda(\xi,\xi_1)=-\eta\left((\xi-\xi_1)^2-|\xi-\xi_1|-\xi^2+|\xi|+|\xi_1|^2-|\xi_1|\right).
\end{align*}
Now, for $\xi \in [2N,2N+4\gamma]$ and $\xi_1 \in K_{\xi}$, it is easy to prove that $\lambda(\xi,\xi_1) \sim \eta N^2.$
Therefore, choosing $t_N=N^{-2-\epsilon}$, it follows that $e^{\eta(|\xi|-\xi^2)t_N}>C_{\eta}>0$. Moreover,
\begin{equation}\label{IPequa8}
\left|\frac{e^{\lambda(\xi,\xi_1)t_N}-1}{\lambda(\xi,\xi_1)}\right|= \frac{1}{N^{2+\epsilon}}+O\left(\frac{1}{N^{2+2\epsilon}}\right).
\end{equation}
Hence, since $mes(K_\xi)\gtrsim \gamma$, it follows from \eqref{IPequa8} that
\begin{align*}
|\widehat{u_2(t_N)}(\xi)|\chi_{[2N,2N+4\gamma]} \gtrsim_{\eta} N^{-2s-1-\epsilon}\chi_{[2N,2N+4\gamma]}. 
\end{align*}
Therefore, we have a lower bound for the norm of $u_2(t_N)$ in $H^s(\mathbb{R})$, given by
\begin{equation}
\left\|u_2(t_N)\right\|_{H^s}^2 \gtrsim_{\eta}  N^{-2s-1-3\epsilon},
\end{equation}
which contradicts \eqref{IPequa7} for $N$ large enough, since $s<-\frac{1}{2}$ and $\left\|\phi\right\|_{H^s}\sim 1$.
\end{proof}

\section{Convergence of solutions of CL to solutions of CLND}

In this section we study the convergence of the solution of the Chen-Lee equation when the dispersion $\beta$ tends to zero and the dissipation $\eta>0$ is fixed. To emphasize the dependence of the semigroup associated with the linear part of the equation \eqref{cl} with the parameter $\beta$, we will use throughout this section the following notation 
$$S_{\beta}(t)\phi =\left(e^{\left(i\beta|\xi|\xi+\eta(|\xi|-\xi^2)\right)t}\widehat{\phi}(\xi) \right)^{\vee},$$
for all $\beta \geq 0$, $\phi\in H^s(\mathbb{R})$, $s\in \mathbb{R}$. 

We observe that the results given in Theorem \ref{mainresult} hold for the initial value problem \eqref{CCLequa1}, since the constants and arguments involved are not dependent upon the parameter $\beta$. In fact, Theorem \ref{mainresult} applies with the same proof when $\beta \in \mathbb{R}$ and $\eta>0$.

\begin{proof}[Proof of the Theorem \ref{convergeCLND}] 
As in the proof of Theorem \ref{mainresult} we consider $\mathfrak{X}_T^s=X_T^s$, if $s\in \left(-\frac{1}{2},0\right)$ and when $s\geq 0$, we take $\mathfrak{X}_T^s=C\left([0,T];H^s(\mathbb{R})\right)$. Let $\beta>0$, then $v^{\beta}=u^{\beta}-u^{0}$ satisfies the integral equation
$$v^{\beta}(t)=\left(S_{\beta}(t)-S_{0}(t)\right)\phi-\frac{1}{2}\int_{0}^t \left(S_{\beta}(t-t')-S_{0}(t-t')\right)\partial_x( (u^{\beta})^2-(u^{0})^2)\ dt'.$$
Therefore, for $0<T_1\leq T$ the triangle inequality implies that 
\begin{align}\label{CCLequa2}
\left\|v^{\beta}\right\|_{\mathfrak{X}_{T_1}^s}\leq & \ \mathcal{I}_{\beta}+\mathcal{II}_{\beta}:= \left\|\left(S_{\beta}(t)-S_{0}(t)\right)\phi\right\|_{\mathfrak{X}_{T_1}^s} \nonumber \\
&+ \left\|\frac{1}{2}\int_{0}^t \left(S_{\beta}(t-t')-S_{0}(t-t')\right)\partial_x( (u^{\beta})^2-(u^{0})^2)\ dt'\right\|_{\mathfrak{X}_{T_1}^s}.
\end{align}
Since the solutions constructed in Theorem \ref{mainresult} satisfy $\left\|u^{\beta} \right\|_{\mathfrak{X}_{\vartheta}^s} \leq \gamma$, for all $\beta \geq 0$, we obtain from Proposition \ref{PropLB2} for $s< 0$, or by Remark \ref{RemarkLB1} when $s\geq 0$, that

\begin{equation}\label{CCLequa3}
\mathcal{II_{\beta}}\leq C_{s,\eta}{T_1}^{g(s)}\left(\left\|u^{\beta}\right\|_{\mathfrak{X}_{T_1}^s}+\left\|u^{0}\right\|_{\mathfrak{X}_{T_1}^s}\right)\left\|v^{\beta}\right\|_{\mathfrak{X}_{T_1}^s}\leq  2C_{s,\eta}{T_1}^{g(s)}\gamma \left\|v^{\beta}\right\|_{\mathfrak{X}_{T_1}^s},
\end{equation}
where $g(s)=\frac{1}{4}(1+2s)$, if $s\in (-\frac{1}{2},0)$, and $g(s)=\frac{1}{4}$, for all $s\geq 0$. So, taking $T_1>0$ such that $T_1 \leq \left( 4C_{s,\eta}\gamma\right)^{-\frac{1}{g(s)}}$ and combining \eqref{CCLequa2} with \eqref{CCLequa3}, we obtain
\begin{equation}\label{CCLequa4}
\frac{1}{2}\left\|v^{\beta}\right\|_{\mathfrak{X}_{T_1}^s}\leq  \mathcal{I}_{\beta}.
\end{equation}
We will estimate $\mathcal{I}_{\beta}$ when $\beta$ tends to zero. Using the mean value inequality and Lemma \ref{LemmaLB1} with $\lambda=1$, we deduce that for all $s>-\frac{1}{2}$ 
\begin{align}\label{CCLequa5}
\left\|\left(S_{\beta}(t)-S_{0}(t)\right)\phi\right\|_{H^s}&=\left\|\left\langle \xi \right\rangle^s e^{\eta(|\xi|-\xi^2)t}\left(e^{i\beta|\xi|\xi t}-1\right)\widehat{\phi}(\xi)\right\|_{L^2(\mathbb{R})}\nonumber \\
&\leq \beta t \left\||\xi|^2 e^{\eta(|\xi|-\xi^2)t}\right\|_{L^{\infty}(\mathbb{R})}\left\|\phi\right\|_{H^s} \nonumber \\
&\lesssim \beta  \left(\frac{\eta T_1+1}{\eta}\right)e^{\frac{\eta}{8}\left(T_1+T_1^{1/2}\sqrt{t+\frac{16}{\eta}}\right)}\left\|\phi\right\|_{H^s}, 
\end{align}
for each $t\in [0,T_1]$. 

On the other hand, since $0\leq t \leq T_1\leq 1$, we have that $t^{1/2}\leq \left\langle \xi t^{1/2} \right\rangle \left\langle \xi \right\rangle^{-1}$ for all $\xi\in \mathbb{R}$. Let $s\in\left(-\frac{1}{2},0\right)$  fixed. Then from the Mean Value Inequality and Lemma \ref{LemmaLB1}, we obtain
\begin{align}\label{CCLequa6}
t^{\frac{|s|}{2}}\left\|\left(S_{\beta}(t)-S_{0}(t)\right)\phi\right\|_{L^2(\mathbb{R})} & \leq  \left\|\left\langle \xi t^{1/2} \right\rangle^{|s|}\left\langle \xi \right\rangle^s e^{\eta(|\xi|-\xi^2)t}\left(e^{i\beta|\xi|\xi t}-1\right)\widehat{\phi}(\xi)\right\|_{L^2(\mathbb{R})} \nonumber \\
&\leq \beta t \left\|\left\langle \xi t^{1/2} \right\rangle^{|s|}|\xi|^2 e^{\eta(|\xi|-\xi^2)t}\right\|_{L^{\infty}(\mathbb{R})}\left\|\phi\right\|_{H^s} \nonumber \\
&\lesssim_s \beta  \left(f_{1}(T_1)+f_{\frac{2+|s|}{2}}(T_1)\right)\left\|\phi\right\|_{H^s},
\end{align}
where $f_{1}$ and $f_{\frac{2+|s|}{2}}$ are defined as in \eqref{LBequa1}. Therefore, from \eqref{CCLequa5}, \eqref{CCLequa6} and the definition of the norm in $\mathfrak{X}_{T_1}^s$, we conclude that $\lim_{\beta \to 0^{+}}\mathcal{I}_{\beta}=0$. Hence, from \eqref{CCLequa4} we deduce that 
$$\sup_{t\in[0,T_1]}\left\|u^{\beta}(t)-u^{0}(t)\right\|_{H^s}\leq  \left\|u^{\beta}-u^{0}\right\|_{\mathfrak{X}_{T_1}^s} \to 0,\text{ as } \beta \to 0^{+}.$$
Finally, we can iterate this process to conclude the result in the whole interval $[0,T]$.
\end{proof}

\section{Convergence of solutions of CL to solutions of BO}

In this section we examine the convergence of solutions of the Chen-Lee equation to solutions of the initial value problem for the integral version of the Benjamin-Ono equation when the dissipation $\eta$ tends to zero and the dispersion $\beta>0$ is fixed. In order to deduce this result we will adapt the ideas employed in the parabolic regularization method (see \cite{BI} and \cite{Duque}).

We will start showing that the time of existence of solutions for CL can be chosen independent of $\eta \in \left(0,1\right)$. 

\begin{lemma}\label{LemmaCBO1}
Let $\phi\in H^s(\mathbb{R})$ where $s> \frac{3}{2}$ is fixed, and let $u^{\eta}\in C\left([0,T];H^s(\mathbb{R})\right)$ be a solution of CL with $\eta>0$. Then there exists a $T'_s>0$ depending on $\left\|\phi\right\|_{H^s}$, but not on $0<\eta<1$, such that $u^{\eta}$ can be extended to the interval $[0,T'_s]$, and there is a function $\rho(t)\in C\left([0,T'_s];\mathbb{R}\right)$ such that
\begin{align}
&\left\|u^{\eta}(t)\right\|_{H^s}\leq \rho(t), \, \rho(0)=\left\|\phi\right\|_{H^s}, \, t\in[0,T'_s].
\end{align}
\end{lemma}  
\begin{proof}
We observe that Kato's Inequality \footnote{This inequality states that for a real value function $u$ in $H^s(\mathbb{R})$ with $s>\frac{3}{2}$, there exists $C_s>0$ depending only on $s$, such that $|\left(u,u\partial_x u\right)_s|\leq C_s\left\|\partial_x u\right\|_{s-1}\left\|u\right\|_{H^s}^2$.} and the assumption that $\eta<1$ imply that there exists a constant depending only on $s$ such that
\begin{align*}
\frac{1}{2}\frac{d}{dt}\left\|u^{\eta}(t)\right\|_{H^s}^2 &=-\left(u^{\eta},u^{\eta}u^{\eta}_x\right)_s-\beta\left(u^{\eta},\mathcal{H}u^{\eta}_{xx}\right)_s-\eta\left(u^{\eta},\mathcal{H}u^{\eta}_{x}-u^{\eta}_{xx}\right)_s \\
& \leq C_s\left\|u^{\eta}(t)\right\|_{H^s}^3+\eta\int_{-\infty}^{\infty}\left(1+\xi^2\right)^s\left(|\xi|-\xi^2\right)|\widehat{u^{\eta}}(\xi)|^2\, d\xi \\
& \leq C_s\left(\left\|u^{\eta}(t)\right\|_{H^s}^3+\left\|u^{\eta}(t)\right\|_{H^s}^2 \right).
\end{align*}
Then, integrating the above expression we get
$$\ln\left(\frac{\left\|u^{\eta}(t)\right\|_{H^s}}{1+\left\|u^{\eta}(t)\right\|_{H^s}}\right)-\ln\left(\frac{\left\|\phi\right\|_{H^s}}{1+\left\|\phi\right\|_{H^s}}\right)=\frac{1}{2}\int_{\left\|\phi\right\|_{H^s}^2}^{\left\|u^{\eta}(t)\right\|_{H^s}^2} \frac{dx}{x^{3/2}+x}\leq \frac{C_s}{2} t.$$
Hence
$$\left\|u^{\eta}(t)\right\|_{H^s}\leq \frac{e^{\frac{C_s}{2}t}\left\|\phi\right\|_{H^s}}{1+\left\|\phi\right\|_{H^s}-e^{\frac{C_s}{2}t}\left\|\phi\right\|_{H^s}}.$$
Therefore defining 
\begin{equation}\label{CBOequa1}
\rho(t)=\frac{e^{\frac{C_s}{2}t}\left\|\phi\right\|_{H^s}}{1+\left\|\phi\right\|_{H^s}-e^{\frac{C_s}{2}t}\left\|\phi\right\|_{H^s}}
\end{equation}
and taking $0<T'_s\leq \frac{2}{C_s}\ln\left(\frac{1+\left\|\phi\right\|_{H^s}}{\left\|\phi\right\|_{H^s}}\right)$ the lemma follows.
\end{proof}
Next we will show that $u^0=\lim_{\eta \to 0^{+}}= u^{\eta}$ exists and satisfies the Benjamin-Ono equation in a weak sense. 
\begin{lemma}\label{LemmaCBO2}
Let $s>\frac{3}{2}$, $\beta>0$ and $\phi \in H^s(\mathbb{R})$, then there exists $T=T\left(s,\left\|\phi\right\|_{H^s}\right)$ and a unique $u^0\in C_w\left([0,T];H^s(\mathbb{R})\right)\cap C_w^1\left([0,T];H^{s-2}(\mathbb{R})\right)$, such that $u^0(0)=\phi$ and 
$$\frac{d}{dt}\left\langle u^{0},\psi\right\rangle_{s-2}=-\left\langle u^0u^0_{x}+ \beta \mathcal{H}u^0_{xx},\psi\right\rangle_{s-2},$$
for all $\psi\in H^s(\mathbb{R})$. Moreover, $\left\|u^0\right\|_{H^s}\leq \rho(t)$ for all $t\in[0,T]$ with $\rho$ defined as in \eqref{CBOequa1}, and $u\in AC\left([0,T];H^s(\mathbb{R})\right)$.  
\end{lemma}
\begin{proof} Let $u^{\eta_1},u^{\eta_2}\in C\left([0,T];H^s(\mathbb{R})\right)$ with $0<\eta_1,\eta_2<1$ be solutions of CL with the same initial data $\phi$, where $T=T\left(s,\left\|\phi\right\|_{H^s}\right)$ is the time of existence given by Lemma \ref{LemmaCBO1} independent of $\eta\in (0,1)$. Let $w=u^{\eta_1}-u^{\eta_2}$, we observe that
\begin{align*}
\frac{1}{2}\left\|w\right\|_{0}^{2}&=-\left(w,u^{\eta_1}u^{\eta_1}_x-u^{\eta_2}u^{\eta_2}_x\right)_0-\beta\left(w,\mathcal{H}w_{xx}\right)-\eta_1\left(w,\left(\mathcal{H}\partial_x-\partial_{xx}\right)w\right)_0\\
& \hspace{12pt}-\left(\eta_1-\eta_2\right)\left(w,\left(\mathcal{H}\partial_x-\partial_{xx}\right)u^{\eta_2}\right)_0 \\
& \lesssim C_s M\left\|u^{\eta_1}-u^{\eta_2}\right\|_0^2 +\eta_1\left\|u^{\eta_1}-u^{\eta_2}\right\|_0^2+|\eta_1-\eta_2|M^2 \\
& \lesssim |\eta_1-\eta_2|M^2+\left(1+C_sM\right)\left\|u^{\eta_1}-u^{\eta_2}\right\|_0^2,
\end{align*}
where $M=\sup_{[0,T]}\rho(t)$, with $\rho$ defined as in Lemma \ref{LemmaCBO1}. Thus integrating the above expression and applying Gronwall's inequality we get
\begin{equation}\label{CBOequa2}
\left\|u^{\eta_1}(t)-u^{\eta_2}(t) \right\|_{0}^2 \lesssim \left|\eta_1-\eta_2\right|M^2Te^{\left(1+C_sM\right)T},
\end{equation}
From \eqref{CBOequa2} we have that there exists $u^{0}\in C\left([0,T];L^2(\mathbb{R})\right)$, such that $\lim_{\eta\to 0^{+}} u^{\eta}=u^{0}$ exists in $L^2(\mathbb{R})$ uniformly over $[0,T]$. 

On the other hand, for $\epsilon>0$ and $\psi\in H^s(\mathbb{R})$, there exist $\psi_{\epsilon}\in \mathcal{S}(\mathbb{R})$ such that $\left\|\psi-\psi_{\epsilon}\right\|_{H^s}<\epsilon$. Then combining \eqref{CBOequa2} and Lemma \ref{LemmaCBO1} it follows that
\begin{align*}
\left|\left\langle u^{\eta_1}-u^{\eta_2},\psi\right\rangle_{H^s}\right| & \leq \left|\left\langle u^{\eta_1}-u^{\eta_2},\psi-\psi_{\epsilon}\right\rangle_{H^s}\right|+\left|\left\langle u^{\eta_1}-u^{\eta_2},\psi_{\epsilon}\right\rangle_{H^s}\right| \\
& \leq \left\|u^{\eta_1}-u^{\eta_2}\right\|_{H^s}\left\|\psi-\psi_{\epsilon}\right\|_{H^s}+\left\|u^{\eta_1}-u^{\eta_2}\right\|_{0}\left\|\psi_{\epsilon}\right\|_{H^{2s}} \\
& \leq 2\epsilon M+C\left(s,\left\|\phi\right\|_{H^s}\right)\left|\eta_1-\eta_2\right|
\end{align*}
so that
$$\lim_{\eta_1,\eta_2 \to 0^{+}} \left\langle u^{\eta_1}-u^{\eta_2},\psi\right\rangle_{H^s}=0.$$
Then, $\left(u^{\eta}\right)_{\eta>0}$ is a weakly Cauchy Net in $H^s(\mathbb{R})$, which converges uniformly over $[0,T]$. Since $H^s(\mathbb{R})$ is reflexive, there exists $\tilde{u^0}\in C_w\left([0,T];H^s(\mathbb{R})\right)$, such that $\lim_{\eta \to 0^{+}}\left\langle u^{\eta},\psi\right\rangle_{H^s}=\left\langle \tilde{u^0},\psi\right\rangle_{H^s}$, for all $\psi \in H^s(\mathbb{R})$ uniformly on $[0,T]$.

From the fact that $\left(L^2(\mathbb{R})\right)'\subseteq \left(H^s(\mathbb{R})\right)'$, we have that $u^{\eta}\rightharpoonup\tilde{u^0}$ as $\eta $ tends to zero in $L^2(\mathbb{R})$, but since strong convergence implies weak convergence, we also have that $u^{\eta}\rightharpoonup u^0$ in $L^2(\mathbb{R})$. Then by uniqueness $\tilde{u^0}(t)=u^{0}(t)$ for each $t\in [0,T]$. Moreover, 
\begin{equation}
\left\|u^0(t)\right\|_{H^s}\leq \sup_{\left\|\psi\right\|_{H^s}=1}\left|\left\langle \lim_{\eta \to 0^{+}} u^{\eta},\psi\right\rangle_{H^s} \right|\leq \rho(t),
\end{equation}
for all $t\in[0,T]$. Next we will prove that $u^0\in C_w^1\left([0,T];H^{s-2}(\mathbb{R})\right)$. Since $u^{\eta}\in H^{\infty}(\mathbb{R})$ for all $t\in \left(0,T\right)$ and $\eta>0$, we have that for all $\psi \in H^{s-2}(\mathbb{R})$,
$$\left\langle u^{\eta},\psi\right\rangle_{s-2}=\left\langle \phi,\psi\right\rangle_{s-2}-\int_{0}^t \left\langle u^{\eta}u^{\eta}_x(t')-Q_{\eta}u^{\eta}(t'),\psi\right\rangle_{s-2} \, dt',$$
where $Q_{\eta}=-\beta\mathcal{H}\partial_{xx}-\eta\left(\mathcal{H}\partial_x-\partial_{xx}\right)$. Since $u^{\eta}\to u^{0}$ in $L^2(\mathbb{R})$ and $u^{\eta}\rightharpoonup u^{0}$ in $H^{s}(\mathbb{R})$, it follows that $Q_{\eta}u^{\eta}\rightharpoonup Q_{0}u^0$ in $H^{s-2}(\mathbb{R})$ and $u^{\eta}\rightharpoonup u^{0}$ in $H^{s-1}(\mathbb{R})$. Then, when $\eta \to 0^{+}$ we deduce 
\begin{align*}
\left\langle u^{0},\psi\right\rangle_{s-2}&=\left\langle \phi,\psi\right\rangle_{s-2}-\int_{0}^t \left\langle u^{0}u^{0}_x(t')-Q_{0}u^{0}(t'),\psi\right\rangle_{s-2} \, dt' \\
&=\left\langle \phi,\psi\right\rangle_{s-2}-\int_{0}^t \left\langle u^{0}u^{0}_x(t')+\beta\mathcal{H}u_{xx}^{0}(t'),\psi\right\rangle_{s-2}\, dt'.
\end{align*}
Therefore, $u^0\in C_w\left([0,T];H^s(\mathbb{R})\right)\cap C_w^1\left([0,T];H^{s-2}(\mathbb{R})\right)$ and satisfies the integral equation weakly. 

Since the application $t\in[0,T]\mapsto u^{0}(t)u^{0}_x(t)+\beta\mathcal{H}u_{xx}^{0}(t)$ is weakly continuous, Bochner-Pettis' Theorem implies that it is strongly measurable. Thus it follows that
\begin{equation}\label{CBOequa3}
u^0(t)=\phi-\int_{0}^t u^{0}u^{0}_x(t')+\beta\mathcal{H}u_{xx}^{0}(t')\, dt',
\end{equation}
which implies that $u^0\in AC\left([0,T];H^{s-2}(\mathbb{R})\right)$.

Finally, to prove the uniqueness, let $v\in C\left([0,T];L^2(\mathbb{R})\right)\cap C_w\left([0,T];H^s(\mathbb{R})\right)\cap C_w^1\left([0,T];H^{s-2}(\mathbb{R})\right)$ be a solution of the Benjamin-Ono equation with $v(0)=\psi$. Since $v$ is strongly differentiable with respect to $t$ in $L^2(\mathbb{R})$, we get 
\begin{align*}
\frac{1}{2}\frac{d}{dt}\left\|u^{0}(t)-v(t)\right\|_0^2&=\left(u^{0}_t-v_t,u^{0}-v\right)_0 \\
& =\frac{1}{4}\left(\partial_x\left(u^0+v\right)_0 ,\left(u^{0}-v\right)^2\right)_0 \\
& \leq \frac{C_s}{2}M \left\|u^{0}(t)-v(t)\right\|_0^2,
\end{align*}
where $M=\sup_{t\in [0,T]}\left\|u^{0}(t)\right\|_{H^s}+\sup_{t\in [0,T]}\left\|v(t)\right\|_{H^s}$. Thus, integrating the above expression and applying Gronwall's inequality we obtain
\begin{align*}
\left\|u^{0}(t)-v(t)\right\|_0^2\leq \left\|\phi-\psi\right\|_{0}^2e^{C_sM T}.
\end{align*}
Uniqueness follows on taking $\phi=\psi$.
\end{proof}

\begin{proof}[Proof Theorem \ref{convergeBO}]

From Lemma \ref{LemmaCBO2} we have that $u^0=\lim_{\eta \to 0^{+}} u^{\eta}$ exists in the class described in this lemma, and is the weak solution of the Benjamin-Ono equation. We claim that $u^0\in C\left([0,T];H^s(\mathbb{R})\right)$. Let $\psi\in H^{s}\left(\mathbb{R}\right)$ be such that $\left\|\psi\right\|_{H^s}=1$. Therefore
$$|\left(\phi,\psi\right)_s|\leq \liminf_{t\to 0^{+}} \left\|u(t)\right\|_{H^s}\leq \limsup_{t\to 0^{+}} \left\|u(t)\right\|_{H^s}\leq \lim_{t\to 0^{+}} \rho(t)=\left\|\phi\right\|_{H^s}.$$
Thus, taking the sup over $\left\|\psi\right\|_{H^s}=1$, we deduce that $\lim_{t\to 0^+}\left\|u(t)\right\|_{H^s}=\left\|\phi\right\|_{H^s}$. Combining this result and the fact that $u^0(t)\rightharpoonup \phi$ in $H^s(\mathbb{R})$, we conclude the continuity at the origin. Let $t^{*}\in \left(0,T\right)$ be fixed, we obtain right continuity at $t^{*}$ from continuity at the origin and uniqueness. Left continuity is deduced using that $u(t^{*}-t,-x)$ is a solution of the Benjamin-Ono equation for any fixed $t^{*}\in(0,T]$. 

On the other hand, since $uu_x+\beta \mathcal{H}u_{xx}\in H^{s-2}(\mathbb{R})$, then by \eqref{CBOequa3} it follows that $u\in C^1\left([0,T];H^{s-2}(\mathbb{R})\right)$. Uniqueness is proved as in Lemma \ref{LemmaCBO2}. Finally, the proof that the solution depends continuously on the initial data is similar to the proof of Theorem 3.1 in \cite{Duque}.
\end{proof}
\section{Decay Properties of the Solution}

In this section we reduce asymptotic questions to existence theorems by solving the equation in appropriate function spaces. As we shall see below the combined effects of the nonlinearity and the non-smoothness of the symbol of the Hilbert transform determine an upper bound for the rate of decay of the solutions of the equation (\ref{cl}) as $|x|\to \infty$. First of all, we will examine the function $u(t)=S(t)\phi$, where $\phi \in \mathcal{F}_{r,r}$ is defined in (\ref{sobpeso}), in order to obtain the semigroup estimates. Next lemma provides some formulas for derivatives of the semigroup associated to the CL equation. They easily follow from a direct computation.
\begin{lemma}\label{lemdecaida}
Let $E(\xi,t)=e^{i\,q(\xi)\,t -p(\xi)\,t}$ where $p(\xi)=\eta (\xi^2-|\xi|)$ and $q(\xi)=\beta \xi |\xi|$. Then,
\begin{align}
\partial_{\xi}E(\xi,t)&=t \bigl[(\eta + 2i\beta \xi )\,sgn(\xi)-2\eta \xi \bigr]E(\xi,t)  \label{uno} \\
\partial_{\xi}^2E(\xi,t)&=2\eta t\delta + 2t\bigl[ i\beta\,sgn(\xi)-\eta \bigl] E(\xi,t) + t^2 \bigl[(\eta + 2i\beta \xi )\,sgn(\xi)-2\eta \xi \bigr]^2E(\xi,t)  \label{dos}\\
\partial_{\xi}^3E(\xi,t)&=2\eta t\delta' + 4i\beta t\delta + 3t^2\bigl[(-2\eta^2-8i\beta \eta \xi)\,sgn(\xi)+2i\beta \eta+4(\eta^2-\beta^2)\xi \bigr]E(\xi,t) +\notag \\
& + t^3 \bigl[(\eta + 2i\beta \xi )\,sgn(\xi)-2\eta \xi \bigr]^3E(\xi,t) \label{tres} \\
\partial_{\xi}^4E(\xi,t)&=2\eta t\delta'' + 4i\beta t\delta' + 2\eta^2t^2(\eta t-6)\delta + 12t^2 \bigl[-2i\beta \eta sgn(\xi)+ \eta^2-\beta^2 \bigl] E(\xi,t) +\notag \\
& +12t^3\bigl[ (i\beta \eta^2+4\eta (\eta^2-\beta^2)\xi+4i\beta (3\eta^2-\beta^2)\xi^2) sgn(\xi) + \notag \\
& \qquad \quad -\eta^3-8i\beta \eta^2 \xi +4\eta(3\beta^2-\eta^2)\xi^2 \bigr]E(\xi,t) +\notag \\
& + t^4\bigl[(\eta + 2i\beta \xi )\,sgn(\xi)-2\eta \xi \bigr]^4E(\xi,t) . \label{cuatro}
\end{align}
Moreover for $j\geq 4$, the $j$-th derivative of $E(\xi,t)$ has the form,
\begin{align}\label{cinco}
\partial_{\xi}^jE(\xi,t)=&2\eta t\delta^{(j-2)}+4i\beta t\delta^{(j-3)}+ \sum_{k=0}^{j-4}p_k(t)\delta^{(k)}+ \sum_{k=0}^{j-1}t^k \bigl[r_k(\xi)\,sgn(\xi)+s_k(\xi) \bigr] E(\xi,t) \notag \\
&+t^j \bigl[(\eta + 2i\beta \xi )\,sgn(\xi)-2\eta \xi \bigr]^jE(\xi,t)
\end{align}
where $\delta$ is the Dirac delta function and $p_k(t)$, $r_k(\xi)$ and $s_k(\xi)$ are polynomials satisfying $deg(p_k(t))\leq j-1$, $deg(r_k(\xi))\leq j-2$ y $deg(s_k(\xi))\leq j-2$.
\end{lemma}
\begin{proof}
This result follows easily from the chain rule.
\end{proof}
\begin{prop} \label{emuc0enfsr}
Let $\eta >0$ and $\beta>0$ be fixed. Then, \\ \\
(a.) $S:[0,+\infty )\longrightarrow \textbf{B}(\mathcal{F}_{r,r})$, $r=0,1$, is a $C^0$-semigroup and satisfies the estimate,
\begin{equation}
\nor{S(t)\phi}{\mathcal{F}_{r,r}} \leq \Bigl(e^{\eta t}\Theta_r(t)+ C_{\eta, \beta} t^{r/2}\Bigr)\,\nor{\phi}{\mathcal{F}_{r,r}}, \label{e6}
\end{equation}
for all $\phi \in \mathcal{F}_{r,r}$, where $\Theta_r(t)$ is a polynomial of degree $r$ with positive coefficients that depend only on $\eta $, $\beta$ and $r$. \\ \\
(b.) If $r\geq 2$ and $\phi \in \mathcal{F}_{r,r}$, the function $S(t)\phi$ belongs to $C([0,\infty);\mathcal{F}_{r,r})$ si, y s\'olo si,
\begin{align}
(\partial_{\xi}^j \widehat{\phi})(0)=0, \qquad \qquad j=0,1,2,\cdots,r-2. \label{e7}
\end{align}
In this case an estimate of the form (\ref{e6}) also holds
\begin{equation}
\nor{S(t)\phi}{\mathcal{F}_{r,r}} \leq \Bigl(e^{\eta t}\Theta_r(t)+ \sum_{l=0}^{3r-2}C_{l,\eta,\beta} \,t^{(l-r+2)/2}\Bigr)\,\nor{\phi}{\mathcal{F}_{r,r}}, \label{e6}
\end{equation}
\end{prop}
\begin{proof}
The proof is similar to the proof of Theorem 2.4 in \cite{Iorio} or Lemma 5.2 in \cite{Borys}.
\end{proof}
\subsection{Theory in $\mathcal{F}_{2,1}(\mathbb{R})$}
Some decay properties of the solution of equation (\ref{cl}) for $\eta>0$ and $\beta>0$ are obtained similarly to those known for the Benjamin-Ono equation (see \cite{Iorio}). Theorem \ref{decaida} is a unique continuation theorem for equation in (\ref{cl}). It implies loss of persistence for CL equation in $\mathcal{F}_{3,3}$, while for the Benjamin-Ono equation this occurs in $\mathcal{F}_{4,4}$. We begin proving a local result for equation (\ref{cl}) in $\mathcal{F}_{2,1}(\mathbb{R})$.
\begin{theorem}
Let $\eta>0$, $\beta >0$ and $\phi \in \mathcal{F}_{2,1}(\mathbb{R})$. Then, there exist $T(\nor{\phi}{\mathcal{F}_{2,1}}, \eta , \beta)>0$ and a unique function $u\in C(0,T;\mathcal{F}_{2,1}(\mathbb{R}))$ satisfying the integral equation
\begin{equation}\label{intequation21}
u(t)=\mathcal{F}_{\eta}(t)\phi - \int_0^t\mathcal{F}_{\eta}(t-t')[u(t')u_x(t')]\,dt' \quad t\geq 0.
\end{equation}
\end{theorem}
\begin{proof}
\end{proof}
\begin{theorem}\label{globalf21}
Let $\phi \in \mathcal{F}_{2,1}(R)$. Then there exists an unique solution $u\in C([0,\infty);\mathcal{F}_{2,1}(\mathbb{R}))$ of the equation (\ref{cl}) such that $\partial_t u \in C(0,\infty; \mathcal{F}_{0,1}(\mathbb{R}))$.  
\end{theorem}
\begin{proof}
To prove global existence for equation (\ref{cl}) in $\mathcal{F}_{2,1}(\mathbb{R})$, it is enough to combine Theorem \ref{x} with the next computations.
\begin{align}
\dfrac{1}{2}\dfrac{d}{dt}\nora{xu(t)}{0}{2}=-(xu,xuu_x)_0-\beta (xu,x\mathcal{H} u_{xx})_0-\eta(xu,x \mathcal{H} u_x)_0+\eta( xu , xu_{xx})_0,
\end{align}
\begin{align}
-(xu,xuu_x)_0&\leq \nor{u_x}{L^{\infty}}\nora{xu}{0}{2}\leq c \nor{u}{2}\nora{u}{L_1^2}{2}, \notag \\
-(xu,x\mathcal{H}u_{xx})_0&=-(xu, \mathcal{H} (xu_{xx}))_0\leq \nor{\mathcal{H} (xu_{xx})}{0}\nor{xu}{0}\leq \nor{xu_{xx}}{0}\nor{u}{L_1^2} \notag \\
-(xu,x \mathcal{H} u_x)_0&= -(xu,\mathcal{H}(xu_x))_0 \leq \nor{xu}{0}\nor{\mathcal{H}(xu_x)}{0}\leq \nor{xu}{0}\nor{xu_x}{0}, \notag \\
(xu,x u_{xx})_0&= \nor{xu}{0}\nor{xu_{xx}}{0} \leq \nor{u}{L_1^2}\nor{xu_{xx}}{0}, \notag
\end{align}
where by (\ref{estpriori3}) in Lemma \ref{estpriori}, $\nor{u(t)}{2}\leq F(t,\eta,\nor{\phi}{2})$, for all $t\in [0,T]$. So,
\begin{equation*}
\dfrac{1}{2}\dfrac{d}{dt}\nora{xu(t)}{0}{2} \leq C_{\eta, \beta} \nora{u(t)}{\mathcal{F}_{2,1}}{2}.
\end{equation*}
Gronwall's inequality then leads the result.
\end{proof}


\subsection{The Unique Continuation Principle}

We will now prove Theorem \ref{decaida}. 
\begin{theorem}
Let $\phi \in \mathcal{F}_{2,2}$. Then for each $\beta$ and $\eta$ positives there exists a unique $u\in C([0,\infty),\mathcal{F}_{2,2})$ such that
$\partial_tu \in C((0,+\infty);\mathcal{F}_{-1,1}(\mathbb{R}))$ and (\ref{cl}) is satisfied.
\end{theorem}

\begin{theorem}\label{ugorroescero}
Let $\beta$, $\eta >0$ be fixed and let $T>0$. Assume that $u \in C([0,T];\mathcal{F}_{2,2}(\mathbb{R}))$ is the solution of (\ref{cl}). Then, $\widehat{u}(t,0)=0$, for all $t\in [0,T]$.
\end{theorem}

\begin{proof}
Multiplying (\ref{cl}) by $x^2$ we obtain
\begin{align}
\partial_t(x^2 u)= -x^2u\partial_xu-\beta x^2\mathcal{H} \partial_x^2u-\eta x^2\mathcal{H}\partial_xu+\eta x^2 \partial_x^2u . \label{x2problema}
\end{align}
By assumption $x^2u(t)\in L^2(\mathbb{R})$, for all $t\in [0,T]$. Then, we have that
\begin{align}
\nor{x^2u\partial_xu}{0}\leq \nor{\partial_xu}{L^{\infty}}\nor{x^2 u}{0}\leq \nor{u}{2}\nor{x^2u}{0} \label{x2problema1}
\end{align}
and therfore $\gamma(t):=x^2(u\partial_xu)(t) \in L^2(\mathbb{R})$, for all $t\in [0,T]$. From this, it follows that $\gamma \in C([0,T];L^2(\mathbb{R}))$, in fact
\begin{align}
\nor{\gamma(t)-\gamma(t_0)}{0}&\leq \nor{x^2u(t)}{0}\nor{\partial_x(u(t)-u(t_0))}{L^{\infty}}+\nor{\partial_xu(t_0)}{L^{\infty}}\nor{x^2(u(t)-u(t_0))}{0} \notag \\
&\leq \nor{u(t)}{\mathcal{F}_{2,2}}\nor{u(t)-u(t_0)}{2}+\nor{u(t_0)}{2}\nor{u(t)-u(t_0)}{\mathcal{F}_{2,2}}. \label{gamacontinua}
\end{align}
Applying the Fourier transform in (\ref{x2problema}) we get
\begin{align}
\partial_t\partial_{\xi}^2\widehat{u}(t,\xi)=\widehat{\gamma(t)}(\xi)+i\beta \partial_{\xi}^2(\,sgn(\xi)\xi^2\widehat{u}(t,\xi))+\eta \partial_{\xi}^2(\,sgn(\xi)\xi \,\widehat{u}(t,\xi)) - \eta \partial_{\xi}^2(\xi^2 \,\widehat{u}(t,\xi)). \label{x2problematf}
\end{align}
Since $u(t)\in \mathcal{F}_{2,2}(\mathbb{R})$ for all $t\in [0,T]$, we have
\begin{align}
\widehat{\alpha(t)}(\xi)&:=\partial_{\xi}^2(\,sgn(\xi)\xi^2\,\widehat{u}(t,\xi)) \notag \\
\widehat{\alpha(t)}(\xi)&= sgn(\xi)(2\widehat{u}(t,\xi)+4\xi\partial_{\xi}\widehat{u}(t,\xi) +\xi^2\partial_{\xi}^2\widehat{u}(t,\xi))\in C([0,T];L_{-2}^2(\mathbb{R})). \label{alpha}
\end{align}
Similarly, we have that
\begin{align}
\partial_{\xi}^2[(\,sgn(\xi)\xi - \xi^2)\widehat{u}(t,\xi)]=2\delta (\xi)\widehat{u}(t,0)+\widehat{\kappa (t)}(\xi), \label{kapa}
\end{align}
where
\begin{equation}
\widehat{\kappa (t)}(\xi)= -2\widehat{u}(t,\xi)+2(\,sgn(\xi)-2\xi) \partial_{\xi}\widehat{u}(t,\xi)+( \,sgn(\xi)\xi - \xi^2) \partial_{\xi}^2\widehat{u}(t,\xi) \in C([0,T];L_{-2}^2(\mathbb{R})). \label{kapa1}
\end{equation}
From (\ref{x2problematf}), (\ref{alpha}), (\ref{kapa}) and (\ref{kapa1}) we have that
\begin{align}
\partial_t\partial_{\xi}^2\widehat{u}(t,\xi)=\widehat{\gamma(t)}(\xi)+i\beta \widehat{\alpha(t)}(\xi)+2\eta \delta (\xi)\widehat{u}(t,0)+\eta \widehat{\kappa (t)}(\xi). \label{comoquedo}
\end{align}

Integrating now (\ref{comoquedo}) between $0$ and $t$, we find that
\begin{align}
2\eta \delta (\xi) \int_0^t\widehat{u}(t',0)\,dt' \in C([0,T];L_{-2}^2(\mathbb{R})),
\end{align}
The last expression implies that
\begin{align}
\int_0^t\widehat{u}(t',0)\,dt'=0,\qquad \text{for all  }t\in [0,T], \label{dacero}
\end{align}
and hence $\widehat{u}(t,0)=0$, for all $t\in [0,T]$.
\end{proof}

We know that
$$u_t=-uu_x - \beta \mathcal{H} u_{xx} - \eta (\mathcal{H}u_x -u_{xx})$$
so,
$$\partial_t\widehat{u}(t,\xi)=\dfrac{1}{\sqrt{2 \pi}}\int_{\mathbb{R}}(-uu_x - \beta \mathcal{H} u_{xx} - \eta (\mathcal{H}u_x -u_{xx}))e^{-i \xi x}\,dx$$
and
$$\partial_t\widehat{u}(t,0)=\dfrac{1}{\sqrt{2 \pi}}\int_{\mathbb{R}}(-uu_x - \beta \mathcal{H} u_{xx} - \eta (\mathcal{H}u_x -u_{xx}))\,dx = 0 ,$$
hence, $\widehat{u}(t,0)$ is a conserved quantity for the problem (\ref{problema}). 
\begin{proof}[Proof of the Theorem \ref{decaida}]
Multiplying (\ref{cl}) by $x^3$ we obtain
\begin{align}
\partial_t(x^3 u)= -x^3u\partial_xu-\beta x^3\mathcal{H} \partial_x^3u-\eta x^3\mathcal{H}\partial_xu+\eta x^3 \partial_x^2u . \label{x3problema}
\end{align}
By assumption $x^3u(t)\in L^2(\mathbb{R})$, for all $t\in [0,T]$. Then, we have that
\begin{align}
\nor{x^3u\partial_xu}{0}\leq \nor{\partial_xu}{L^{\infty}}\nor{x^3 u}{0}\leq \nor{u}{2}\nor{x^3u}{0}. \label{x3problema1}
\end{align}
So, $\gamma(t):=x^3(u\partial_xu)(t) \in L^2(\mathbb{R})$, for all $t\in [0,T]$. Similar to Theorem \ref{ugorroescero} we have that $\gamma \in C([0,T];L^2(\mathbb{R}))$. Taking the Fourier transform in (\ref{x3problema}) we find that
\begin{align}
\partial_t\partial_{\xi}^3\widehat{u}(t,\xi)=-i\widehat{\gamma(t)}(\xi)+i\beta \partial_{\xi}^3(\,sgn(\xi)\xi^2 \, \widehat{u}(t,\xi))+ \eta \partial_{\xi}^3(\,sgn(\xi)\xi \, \widehat{u}(t,\xi)) -\eta \partial_{\xi}^3(\xi^2\,\widehat{u}(t,\xi)). \label{x3problematf}
\end{align}
We see that
\begin{align}
\partial_{\xi}^3(\,sgn(\xi)\xi^2\widehat{u}(t,\xi))&=4\delta(\xi)\widehat{u}(t,0)+6 sgn(\xi)\partial_{\xi}\widehat{u}(t,\xi)+6 \xi sgn(\xi) \partial_{\xi}^2\widehat{u}(t,\xi)+ sgn(\xi) \xi^2 \partial_{\xi}^3\widehat{u}(t,\xi) \notag \\
&=4\delta(\xi)\widehat{u}(t,0)+\widehat{\alpha (t)}(\xi) \label{x3problema2}
\end{align}
where 
\begin{equation}\label{x3problema3}
\widehat{\alpha (t)}(\xi)=6 sgn(\xi)\partial_{\xi}\widehat{u}(t,\xi)+6 \xi sgn(\xi) \partial_{\xi}^2\widehat{u}(t,\xi)+ sgn(\xi) \xi^2 \partial_{\xi}^3\widehat{u}(t,\xi) \in C([0,T];L_{-2}^2(\mathbb{R})), 
\end{equation}
and
\begin{align}
\partial_{\xi}^3[(\,sgn(\xi) \xi - \xi^2)\widehat{u}(t,\xi)]&=2\delta'(\xi)\widehat{u}(t,0)+4 \delta(\xi)\partial_{\xi}\widehat{u}(t,0)+\widehat{\Gamma (t)}(\xi), \label{x3problema4}
\end{align}
where
\begin{align}
\widehat{\Gamma (t)}(\xi)&=3\,sgn(\xi) \partial_{\xi}^2\widehat{u}(t,\xi)+\,sgn(\xi) \xi \partial_{\xi}^3\widehat{u}(t,\xi)-6\,\partial_{\xi} \widehat{u}(t,\xi)-6\xi \partial_{\xi}^2\widehat{u}(t,\xi) +\notag \\
&-\xi^2 \partial_{\xi}^3 \widehat{u}(t,\xi)\in C([0,T];L_{-2}^2(\mathbb{R})). \label{x3problema5}
\end{align}
From (\ref{x3problematf})-(\ref{x3problema5}) we get
\begin{align}
-i\partial_t \partial_{\xi}^3\widehat{u}(t,\xi)=&-\widehat{\gamma(t)}(\xi)+4\beta \delta(\xi) \widehat{u}(t,0) + \beta \widehat{\alpha (t)}(\xi) - i \eta \widehat{\Gamma(t)}(\xi) \notag \\
& -2i \eta \delta'(\xi)\widehat{u}(t,0)-4 i \eta \delta(\xi)\partial_{\xi}\widehat{u}(t,0). \label{x3problematf1}
\end{align}
Integrating (\ref{x3problematf1}) between $0$ and $t$, we have that
\begin{align}
(4\beta \delta(\xi)-2i \eta \delta'(\xi))\int_0^t\widehat{u}(t',0)\,dt'-4 i \mu \delta(\xi)\int_0^t\partial_{\xi}\widehat{u}(t',0)\,dt' \in C([0,T];L_{-2}^2(\mathbb{R})). \label{x3problematf2}
\end{align}
Then,
$$\int_0^t\widehat{u}(t',0)\,dt'=\int_0^t\partial_{\xi}\widehat{u}(t',0)\,dt'=0,$$
for all $t\in [0,T]$. The last expression implies that 
\begin{align}
\widehat{u}(t,0)=\partial_{\xi}\widehat{u}(t,0)=0,\label{x3problema4}
\end{align}
for all $t\in [0,T]$.

On the other hand, we have that $u$ satisfies the integral equation 
\begin{align}
u(t,\cdot )=S(t)\phi(\cdot) -\dfrac{1}{2}\int_0^t S(t-\tau)\,\partial_x(u^2)(\tau ,\cdot)\,d\tau  \label{ecintegral1}
\end{align}
for $t\in [0,T]$. Denoting $v:=u^2$, $w:=\partial_xv$ and taking the Fourier transform in (\ref{ecintegral1}) we get
\begin{align}
\widehat{u}(t,\xi )&=E(t,\xi )\widehat{\phi}(\xi) -\dfrac{1}{2}\int_0^t E(t-\tau ,\xi )\,\widehat{w}(\tau ,\xi)\,d\tau . \label{secalcula} \\
\end{align}
Derivating three times the equation (\ref{secalcula}), respect to $\xi $, we obtain
\begin{align}
\partial_{\xi}^3\widehat{u}(t,\xi )&=\partial_{\xi}^3E(t,\xi)\widehat{\phi}(\xi)+3\partial_{\xi}^2E(t,\xi)\partial_{\xi}\widehat{\phi}(\xi)+3\partial_{\xi}E(t,\xi)\partial_{\xi}^2 \widehat{\phi}(\xi)\notag \\
&+E(t,\xi)\partial_{\xi}^3 \widehat{\phi}(\xi)-\dfrac{1}{2}\int_0^t  \partial_{\xi}^3E(t-\tau ,\xi)\widehat{w}(\tau ,\xi)\,d\tau \notag \\
&-\dfrac{3}{2}\int_0^t\partial_{\xi}^2E(t-\tau ,\xi)\partial_{\xi}\widehat{w}(\tau ,\xi)\,d\tau -\dfrac{3}{2}\int_0^t\partial_{\xi}E(t-\tau ,\xi)\partial_{\xi}^2\widehat{w}(\tau ,\xi)\,d\tau \notag \\
&- \dfrac{1}{2}\int_0^t E(t-\tau ,\xi)\partial_{\xi}^3 \widehat{w}(\tau ,\xi)\,d\tau .\label{trexider}
\end{align}
Since $\eta >0$, $\phi , u ,v \in \mathcal{F}_{3,3}$ and using Lemma (\ref{lemdecaida}), it follows that:
\begin{align}
\partial_{\xi}E(t,\xi)\partial_{\xi}^2 \widehat{\phi}(\xi)=t[(\eta + 2i\beta \xi )\,sgn(\xi)-2\eta \xi ] E(t,\xi)\partial_{\xi}^2 \widehat{\phi}(\xi) \in C([0,T];L^2(\mathbb{R})) \label{termino3}
\end{align}
and
\begin{align}
E(t,\xi)\partial_{\xi}^3 \widehat{\phi}(\xi) \in C([0,T];L^2(\mathbb{R})). \label{termino4}
\end{align}
We can see that
\begin{align}
\partial_{\xi}^2\widehat{w}(\tau ,\xi)= i(2\partial_{\xi}\widehat{v}(\tau ,\xi)+ \xi \partial_{\xi}^2\widehat{v}(\tau ,\xi)) \notag
\end{align}
and
\begin{align}
\partial_{\xi}^3\widehat{w}(\tau ,\xi)= i(3\partial_{\xi}^2\widehat{v}(\tau ,\xi)+ \xi \partial_{\xi}^3\widehat{v}(\tau ,\xi)). \notag
\end{align}
Then,
\begin{align}
-\dfrac{3}{2}\int_0^t\partial_{\xi}E(t-\tau ,\xi)\partial_{\xi}^2\widehat{w}(\tau ,\xi)\,d\tau\in C([0,T];L^2(\mathbb{R})), \label{termino7}
\end{align}
and
\begin{align}
-\dfrac{1}{2}\int_0^t E(t-\tau ,\xi)\partial_{\xi}^3\widehat{w}(\tau ,\xi)\,d\tau \in C([0,T];L^2(\mathbb{R})). \label{termino8}
\end{align}
Similarly, we find that
\begin{align}
\partial_{\xi}^3E(t,\xi)\widehat{\phi}(\xi)=f_1(t,\xi)+ 4it\delta(\xi)\widehat{\phi}(\xi) + 2 \mu t \delta'(\xi)\widehat{\phi}(\xi), \label{termino1}
\end{align}
where $f_1(t,\xi )\in C([0,T];L^2(\mathbb{R}))$.

Using (\ref{termino3})-(\ref{termino1}) and making similar considerations to the other terms in (\ref{trexider}), we obtain that
\begin{align}
\partial_{\xi}^3\widehat{u}(t,\xi )&= f(t,\xi)+2\eta t \delta'(\xi)\widehat{\phi}(\xi)+4i \beta t \delta(\xi)\widehat{\phi}(\xi)+6\eta t\delta(\xi)\partial_{\xi}\widehat{\phi}(\xi) \notag \\
&\quad -\dfrac{1}{2}\int_0^t2\eta (t-\tau )\delta'(\xi)\widehat{w}(\tau , \xi)\,d\tau - \dfrac{1}{2}\int_0^t 4i \beta (t-\tau) \delta(\xi) \widehat{w}(\tau ,\xi)\,d\tau \notag \\
&\quad -\dfrac{3}{2}\int_0^t2\eta (t-\tau )\delta(\xi)\partial_{\xi}\widehat{w}(\tau , \xi)\,d\tau \label{reduciendo}
\end{align}

where $f(\cdot , \xi)\in C([0,T];L^2(\mathbb{R}))$. Since 
$$\delta(\xi)\partial_{\xi}\widehat{\phi}(\xi)=\delta(\xi)\partial_{\xi}\widehat{\phi}(0)=0,$$
we have that
\begin{align}
\partial_{\xi}^3\widehat{u}(t,\xi )&= f(t,\xi)+(2\eta t \widehat{\phi}(\xi)- \eta \int_0^t (t-\tau )\widehat{w}(\tau , \xi)\,d\tau )\delta'(\xi) \notag \\
&+(4i\beta t\widehat{\phi}(\xi)-2i\beta \int_0^t (t-\tau)\widehat{w}(\tau ,\xi)\,d\tau -3\eta \int_0^t (t-\tau )\partial_{\xi}\widehat{w}(\tau , \xi)\,d\tau )\delta(\xi) \label{reduciendo1}
\end{align}

Since $\partial_{\xi}^3\widehat{u}(t,\xi )$ and $f(t,\xi)$ are measurable functions for all $t\in [0,T]$, it follows from equation (\ref{reduciendo1}) that

\begin{align}
2\eta t \widehat{\phi}(\xi)- \eta \int_0^t (t-\tau )\widehat{w}(\tau , \xi)\,d\tau &=0 \label{yacasi1} \\
4i\beta t\widehat{\phi}(\xi)-2i\beta \int_0^t (t-\tau)\widehat{w}(\tau ,\xi)\,d\tau -3\eta \int_0^t (t-\tau )\partial_{\xi}\widehat{w}(\tau , \xi)\,d\tau &=0 .  \label{yacasi2}
\end{align}

But $\widehat{w}(\tau ,0)=\widehat{\partial_xv}(\tau , 0)=0$ and $\widehat{\phi}(0)=0$ from (\ref{x3problema4}) then

\begin{align}
\int_0^t (t-\tau )\partial_{\xi}\widehat{w}(\tau , 0)\,d\tau &=0 , \qquad \text{for all}\quad t\in [0,T] \label{yacasi3}
\end{align}

Let $t\in [0,T]$. Since $u(t)\in \mathcal{F}_{3,3}$ we can see that $xu(t)\in L_2^2(\mathbb{R})$. Then, $\widehat{xu(t)}\in H^2(\mathbb{R})$ and therfore $xu(t,\cdot)\in L^1(\mathbb{R})$. So,

\begin{align}
\int |x\partial_xu(t)^2|\,dx &= 2 \int |xu(t)\partial_xu(t)|\,dx \leq 2\nor{\partial_xu(t)}{L^{\infty}}\int |xu(t,x)|\,dx \notag \\
&\leq 2\nor{u(t)}{2}\nor{xu(t,\cdot)}{L^1}< +\infty.
\end{align}
Then,
\begin{align}
i\partial_{\xi}\widehat{w}(t,\xi)=\widehat{xw}(t,\xi)=\dfrac{1}{\sqrt{2\pi}}\int_{\mathbb{R}}x(\partial_xu^2)(t,x) e^{-i\xi x}\,dx , \notag
\end{align}
and,
\begin{align}
\partial_{\xi}\widehat{w}(t,0)=- \dfrac{i}{\sqrt{2 \pi}}\int_{\mathbb{R}}x (\partial_xu^2)(t,x)\,dx = \dfrac{1}{\sqrt{2 \pi}}\nor{u(t)}{0}^2 .\label{yacasi4}
\end{align}

Combining (\ref{yacasi3}) and (\ref{yacasi4}) we get
\begin{align}
\int_0^t(t-\tau)\nora{u(\tau)}{0}{2}\,d\tau =0, \label{listo}
\end{align}

for all $t\in [0,T]$. From (\ref{listo}), we can conclude that $\nor{u(t)}{0}=0$, for all $t\in [0,T]$. This completes the proof of the theorem.

\end{proof}

\subsection*{Acknowledgements}
The authors are supported by the Universidad Nacional de Colombia, sede Bogot\'a. The authors would like to thank the seminar on Evolution PDE for postgraduate students at the Universidad Nacional de Colombia, sede Bogot\'a, for all their help and comments.


\bibliographystyle{amsplain}

\end{document}